\documentclass[reqno,a4paper]{amsart}

\usepackage{mathrsfs,amsmath,amssymb,graphicx}
\usepackage{stmaryrd}
\usepackage[unicode]{hyperref}
\usepackage{latexsym}
\usepackage{layout}
\usepackage[english]{babel}
\usepackage{color}
\usepackage{fancyvrb}
\usepackage{color}
\usepackage{amsmath,amscd}
\usepackage{multicol} 
\usepackage{enumitem}
\usepackage{subcaption}
\usepackage[percent]{overpic} 
\usepackage{transparent}

\usepackage{tikz}

\usetikzlibrary{calc}
\usetikzlibrary{matrix,arrows,decorations}

\theoremstyle{plain}
\newtheorem*{theorem*}{Theorem} 
\newtheorem{theorem}{Theorem}[section]
\newtheorem{lemma}[theorem]{Lemma}
\newtheorem{corollary}[theorem]{Corollary} 
\theoremstyle{definition}

\theoremstyle{remark}

\newcommand{\HK}{{\rm HK}}

\newcommand{\pair}{(\sphere,\HK)}

\newcommand{\pairV}{(\sphere,V)}
\newcommand{\pairVn}{(\sphere,V_n)}
\newcommand{\pairVmn}{(\sphere,V_{m,n})}
\newcommand{\Z} {\mathbb Z}

\newcommand{\id}{{\rm id}}

\newcommand{\Sbb} {\mathbb S}
\newcommand{\bd}[1]{{\boldsymbol #1}}
\newcommand{\sphere} {\Sbb^3}
\newcommand{\Stwo}{\mathcal{S}}
\newcommand{\p}[1]{{P}_{#1}}
\newcommand{\systemP}{\mathbf{P}}
\newcommand{\systemQ}{\mathbf{Q}}
\newcommand{\uniP}{\mathcal{P}}
\newcommand{\uniQ}{\mathcal{Q}}
\newcommand{\uniD}{\mathcal{D}}
\newcommand{\one}{\mathtt{1}}
\newcommand{\two}{\mathtt{2}}
\newcommand{\three}{\mathtt{3}}
\newcommand{\alphaQ}{\bd \alpha_{\rm Q}}
\newcommand{\alphaP}{\bd \alpha_{\rm P}}
\newcommand{\alphaA}{\bd \alpha_{\rm A}}

\newcommand{\bn}{{\bd n}}

\newcommand{\tind}{(n_\partial,n_\circ)}
\newcommand{\bdtind}{(\bd n_\partial,\bd n_\circ)}
\newcommand{\pind}{(n_\partial,n_\otimes)}
\newcommand{\pindtilde}{(\tilde n_\partial,\tilde n_\otimes)}

\newcommand{\Compl}[1]{E(#1)}

\newcommand{\rnbhd}[1]{\mathfrak N(#1)} 
\newcommand{\fron}{\partial_f}

\newcommand{\Sym}[2][\sphere]{\mathcal{MCG}(#1, #2)}
\newcommand{\pSym}[2][\sphere]{\mathcal{MCG}_+(#1, #2)}

\newcommand{\Aut}[1]{{\mathcal Homeo}(#1)}
\newcommand{\pAut}[1]{{\mathcal Homeo_+}(#1)}
\newcommand{\nAut}[1]{{\mathcal Homeo_-}(#1)} 
\newcommand{\Emb}[1]{{\mathcal Emb}_0(#1)}

\newcommand{\rel}{{\rm rel\,}}

\newcommand{\MCG}[1]{\mathcal{MCG}(#1)}
\newcommand{\pMCG}[1]{\mathcal{MCG}_+(#1)}

\newcommand{\fourone}{\mathbf{4_1}}
\newcommand{\sixone}{\mathbf{6_1}}
\newcommand{\sixten}{\mathbf{6_{10}}}
\newcommand{\twoone}{\mathbf{2_1}} 
\newcommand{\wheel}{W^g}
\newcommand{\pairwheel}{(\sphere,\wheel)}
\newcommand{\hpairn}{(\sphere,H_n)}
\newcommand{\hpairm}{(\sphere,H_m)} 

\newcommand{\sghk}{\Gamma_\HK}

\newcommand{\sgX}{\Gamma_X}
\newcommand{\sgY}{\Gamma_Y}

\newcommand{\hkX}{\HK_X}
\newcommand{\hkY}{\HK_Y}
\newcommand{\hkx}{\HK^X}
\newcommand{\inducedhkX}{(\sphere,\hkX)}
\newcommand{\inducedhkY}{(\sphere,\hkY)}
\newcommand{\inducedhkx}{(\sphere,\hkx)}

\newcommand{\inducedsgY}{(\sphere,\Gamma_Y)}
\newcommand{\inducedsgX}{(\sphere,\Gamma_X)}

\newcommand{\nada}[1]   {}
\newcommand{\cout}[1]   {}

\definecolor{mygray}{rgb}{0.92,0.92,0.92}

\graphicspath{{figures/}}

\numberwithin{equation}{section}
\numberwithin{figure}{section}

\newif\ifdraft
\draftfalse

\title{On unique decomposition of knotted handlebodies}
\author{Giovanni Bellettini}
\address{Dipartimento di Ingegneria dell'Informazione e Scienze Matematiche, Universit\`a di Siena, 53100 Siena, Italy,
and International Centre for Theoretical Physics ICTP,
Mathematics Section, 34151 Trieste, Italy
}
\email{giovanni.bellettini@unisi.it}
\author{Maurizio Paolini}
\address{Dipartimento di Matematica e Fisica, Universit\`a Cattolica del Sacro Cuore, 25121 Brescia, Italy}
\email{maurizio.paolini@unicatt.it}
\author{Yi-Sheng Wang}
\address{Department of Applied Mathematics, National Sun Yat-sen University}
\email{yisheng@math.nsysu.edu.tw}

\date{\today}

\begin{document}

\thanks{Y.-S. W. gratefully acknowledges the support from NSTC, Taiwan (grant no. 110-2115-M-001-004-MY3 and 112-2115-M-110 -001 -MY3 ).}
\
\begin{abstract}
The paper considers the uniqueness question of factorization of a knotted handlebody in the $3$-sphere along decomposing $2$-spheres. We obtain a uniqueness result for factorization along decomposing $2$-spheres meeting the handlebody at three parallel disks. The result is used to examine handlebody-knot symmetry; particularly, the chirality of $\sixten$ in the handlebody-knot table, previously unknown, is determined. In addition, an infinite family of hyperbolic handlebody-knots with homeomorphic exteriors is constructed. 
\end{abstract}

\maketitle
 
\section{Introduction}\label{sec:intro}
A fundamental result by Schubert \cite{Sch:49} 
asserts that every knot can be split, along $2$-spheres meeting the knot at \emph{two} points, into prime knots in a unique way. Schubert's result is subsequently generalized to links by Hashizume \cite{Has:58}. More generally, one may split a knot by $2$-spheres meeting a knot at \emph{$n$} points; this leads to the notion of tangle decomposition and $n$-string prime; see for instance, Gordon-Reid \cite{GorRei:95}, Ozawa \cite{Oza:98}. 
In the context of spatial graphs, Suzuki \cite{Suz:87} shows that a connected
spatial graph can be factorized, 
along $2$-spheres meeting the graph at \emph{one or two} points, into some prime spatial and trivial graphs in a unique manner. 
Motohashi \cite{Moto:98}, \cite{Moto:07}, 
considering $2$-spheres meeting the spatial graph at \emph{three} points, also obtains unique factorizations for various types of spatial graphs.

Aforementioned unique factorization theorems assert the prime factors are unique, but the $2$-spheres splitting the knot or spatial graph are in general not unique. However, adding constraints to splitting $2$-spheres may lead to useful uniqueness result.
For example, Flapan-Mellor-Naimi \cite{FlaMelNai:12}, considering balls meeting a spatial graph 
in an arc that contain all local knots, obtains a uniqueness theorem for such decomposing balls, which is then used to analyze spatial graph symmetry.  
In the context of handlebody-knots, Ishii-Kishimoto-Ozawa \cite{IshKisOza:15} and Koda-Ozawa \cite{KodOzaGor:15} 
study decomposing $2$-spheres that meet the handlebody-knot at \emph{two parallel disks}, and show that there is 
a unique maximal unnested set of such $2$-spheres for every handlebody-knot, up to certain moves; their result is employed to detect inequivalence and chirality of genus two handlebody-knots in the Ishii-Kishimoto-Moriuchi-Suzuki handlebody-knot table \cite{IshKisMorSuz:12}. 
The present paper concerns decomposition of handlebody-knots 
via $2$-spheres meeting the handlebody at \emph{three parallel disks}. 

  
\subsection*{Preliminaries}
A genus $g$ \emph{handlebody-knot} $\pair$ is an embedding of a genus $g$ handlebody $\HK$ in the oriented $3$-sphere $\sphere$.
Recall from Ishii-Kishimoto-Ozawa \cite{IshKisOza:15} 
that an \emph{$n$-decomposing sphere $\Stwo$} of $\pair$ is a $2$-sphere $\Stwo\subset \sphere$ that meets $\HK$ at $n$ disjoint disks so that the $n$-punctured sphere $P:=\Stwo\cap \Compl\HK$ in the exterior $\Compl\HK:=\overline{\sphere-\HK}$ is incompressible and non-boundary parallel. 
Conversely, an incompressible, non-boundary parallel, meridional $n$-punctured sphere $P\subset\Compl\HK$ 
gives rise to an $n$-decomposing sphere; 
thus we call $P$ an \emph{$n$-decomposing surface}.

A handlebody-knot is \emph{$n$-decomposable} if it admits an $n$-decomposing sphere. When $n=1$, this leads to the prime decomposition in Suzuki \cite{Suz:75}. In the case $n>1$ and $g>2$, the notion of $n$-decomposability is too broad 
to hope for a uniqueness result. For instance, 
Figs.\ \ref{fig:decomposition_i}, 
\ref{fig:decomposition_ii} 
illustrate two ways to decompose the 
handlebody-knot $\fourone$ in the handlebody-knot 
table \cite{IshKisMorSuz:12}; also, a trivial handlebody-knot of genus $3$ is $n$-decomposable for every $n\geq 2$; see Fig.\ \ref{fig:trivial_n_decomposition}, for $n=2,3,4$. 
On the other hand, imposing additional conditions on $\partial P$ in relation to $\HK$ can lead to useful uniqueness results. 
\begin{figure}
\begin{subfigure}[b]{.32\linewidth}
\centering
\begin{overpic}[scale=.12,percent]{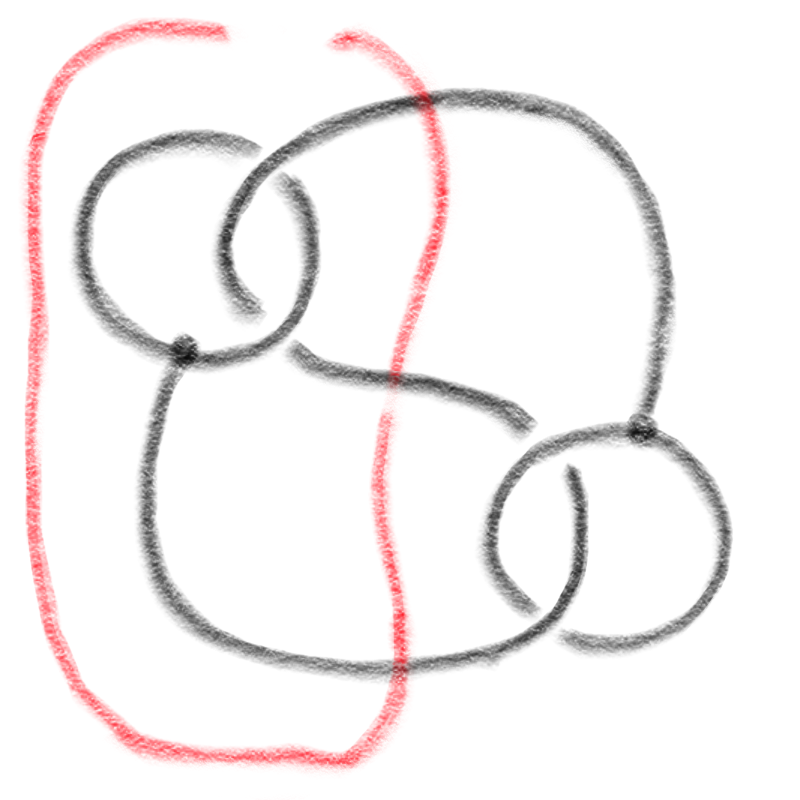}
\put(30,92){$P$}
\end{overpic}
\caption{Decomposing $\fourone$ I.}
\label{fig:decomposition_i}
\end{subfigure}
\begin{subfigure}[b]{.32\linewidth}
\centering
\begin{overpic}[scale=.12,percent]{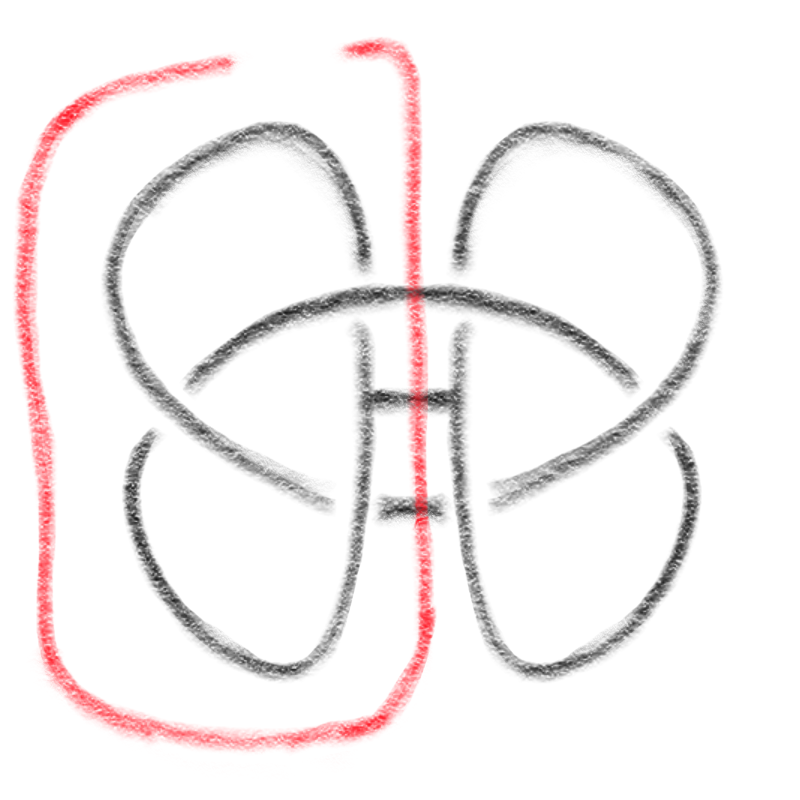}
\put(32,90){$P$}
\end{overpic}
\caption{Decomposing $\fourone$ II.}
\label{fig:decomposition_ii}
\end{subfigure}
\begin{subfigure}[b]{.32\linewidth}
\centering
\begin{overpic}[scale=.13,percent] {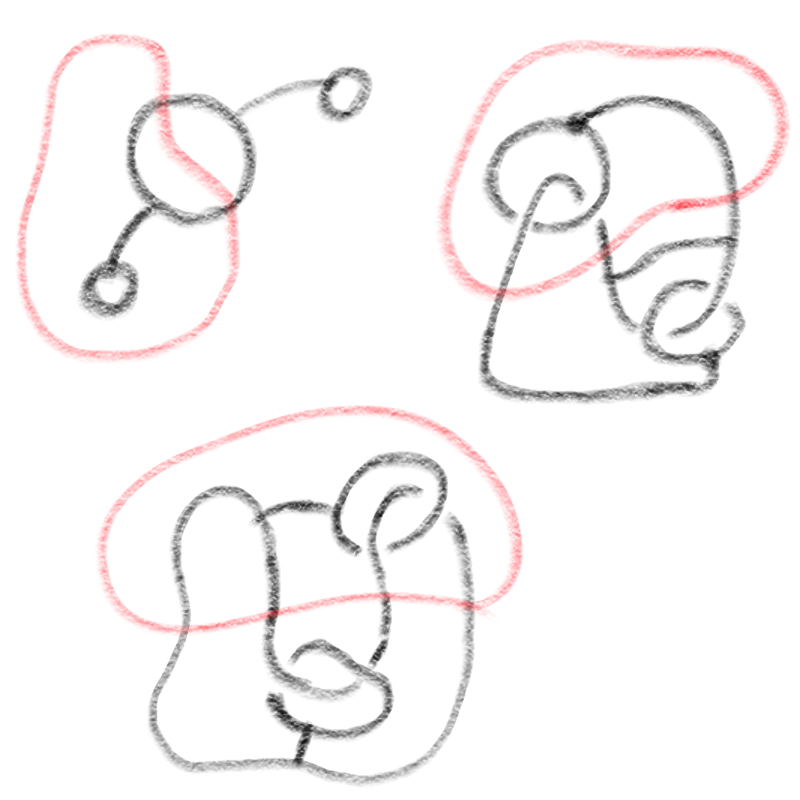}
\put(27,55){\footnotesize $n=2$}
\put(80,43){\footnotesize $n=3$}
\put(62,19){\footnotesize $n=4$}
\end{overpic}
\caption{$g=3$; $n=2,3,4$.}
\label{fig:trivial_n_decomposition}
\end{subfigure}
\caption{}
\end{figure}
 
A \emph{$\p n$-sphere} is an $n$-decomposing sphere $\Stwo$ with 
disks in $\Stwo\cap\HK$ mutually parallel in $\HK$ (see Fig.\ \ref{fig:decomposition_i} for $n=3$).
A \emph{$\p n$-surface} is an $n$-decomposing surface $P$ with 
components of $\partial P$ mutually parallel in $\partial\HK$.     
Every $\p n$-sphere induces a $\p n$-surface and vice versa. The handlebody-knot $\pair$ is said to be \emph{$\p n$-decomposable}
if $\pair$ admits a $\p n$-decomposing sphere and otherwise \emph{$\p n$-indecomposable}. 
Note that $\p 1$-decomposability is the same as $1$-decomposability,
whereas $\p 2$-decomposability is equivalent to $2$-decomposability only when $g\leq 2$. 
In addition, if $g=2$, $\Compl\HK$ is $\partial$-reducible 
if and only if 
$\pair$ is $\p 1$-decomposable by Tsukui \cite{Tsu:75}, yet the assertion fails when $g>2$ by Suzuki \cite{Suz:75}. 
We remark that a $\p 2$-decomposing sphere is precisely the \emph{knotted handle decomposing sphere} in Ishii-Kishimoto-Ozawa 
\cite{IshKisOza:15}.

A \emph{$\p n$-system} $\systemP$ of $\pair$ is a set of disjoint, mutually 
non-parallel
$\p n$-surfaces, and $\systemP$ is 
\emph{maximal}
if $\systemP$ is not contained in another $\p n$-system as a proper subset. 
The result of cutting $\Compl\HK$ along the $\p n$-surfaces in $\systemP$ is called a $\p n$-decomposition of $\pair$.

When $g=2$, the uniqueness of a maximal $\p 1$-system of $\pair$ follows from Tsukui \cite{Tsu:70}, yet for general $g$, the uniqueness remains open; see Suzuki \cite{Suz:75}. On the other hand, it follows from Ishii-Kishimoto-Ozawa \cite{IshKisOza:15} and Koda-Ozawa \cite{KodOzaGor:15} that a maximal unnested $\p 2$-system of $\pair$ is unique, 
up to annulus move. 
\subsection*{Main Results} 
In this paper we prove a uniqueness theorem for a maximal $\p 3$-system of $\pair$ and use it to
study the symmetry of $\p 3$-decomposable handlebody-knots. A topological characterization of such handlebody-knots is also given.
\begin{theorem}[{Theorem \ref{teo:uniqueness}}]\label{intro:teo:uniqueness}
Given a $\p 3$-decomposable handlebody-knot $\pair$, if $\pair$ is $2$-indecomposable and the exterior $\Compl\HK$ is $\partial$-irreducible, then, up to isotopy, $\pair$ admits a unique maximal $\p 3$-system $\systemP$.
\end{theorem}
If $g=2$, the $2$-indecomposability can be replaced with $\p 2$-indecomposability and $\partial$-irreducibility can be dropped. 
The topology of $\p 3$-decomposable handlebody-knots is examined in Section \ref{sec:topology}, and the following is obtained.
\begin{theorem}\label{intro:teo:topology} 
Let $\systemP$ be a $\p 3$-system of $\pair$, and $\uniP$ the union of its members.      
\begin{enumerate}
\item If $\Compl \HK$ admits an essential disk (resp.\ incompressible torus), then it admits an essential disk (resp.\ incompressible torus) disjoint from $\uniP$. 
\item If $\Compl \HK$ admits an essential annulus, then it admits an essential disk or annulus disjoint from $\uniP$.
\end{enumerate}
\end{theorem}
Theorem \ref{intro:teo:topology}
implies that, when $g=2$, $\Compl\HK$ is always $\partial$-irreducible (see Corollary \ref{cor:genus_two_b_irre}). It also allows us to produce an infinite family of hyperbolic handlebody-knots with homeomorphic exteriors (see Fig.\ \ref{fig:hyp_n}) by Thurston's hyperbolization theorem. Handlebody-knots in the infinite families in Motto \cite{Mott:90} and Lee-Lee \cite{LeeLee:12} all admit essential annuli in their exteriors, and hence are non-hyperbolic. On the other hand, the argument in Section \ref{subsubsec:hyp} provides an alternative way to see the inequivalence of Motto's handlebody-knots, which are all $\p 3$-decomposable. 

The symmetry of a handlebody-knot can be measured by the \emph{symmetry group}: 
$\Sym\HK:=\pi_0(\Aut{\sphere,\HK})$, where $\Aut{\sphere,\HK}$ is the space of p.l. homeomorphisms of $\pair$. A great deal of work has been done for genus one handlebody-knots, equivalent to knots, and the symmetry group structure has been determined for many classes of knots (see Kawasaki \cite{Kaw:96}). Notably, it is now known that if it is finite, then $\Sym\HK$ is either cyclic or dihedral. Much less is understood in the higher genus case. Funayoshi-Koda \cite{FunKod:20}, however, shows that when $g=2$, $\Sym\HK$ is finite if and only if $\Compl\HK$ contains no incompressible torus, namely \emph{atoroidal}. Efforts have been made to 
study the finite symmetry group structure (see \cite{Wan:21}, \cite{Wan:23}, \cite{Wan:24}), yet a classification is still to be obtained. For the toroidal case, see Koda \cite{Kod:15}. We investigate the symmetry of a $\p 3$-decomposable handlebody-knot in Section \ref{sec:applications} and prove the following.    
\begin{theorem}[{Theorem \ref{teo:genus_two_symmetry}}]\label{intro:teo:symmetry_two}
Given a $\p 3$-decomposable handlebody-knot $\pair$ of genus $g=2$, if $\Compl\HK$ is atoroidal, then 
\[\Sym\HK<\Z_2\oplus \Z_2,\] 
and 
$\Sym\HK\simeq \Z_2\oplus \Z_2$ 
if and only if $\pair$ is equivalent to $\fourone$ in 
the handlebody-knot table 
\cite{IshKisMorSuz:12}. 
\end{theorem}
Given a $\p 3$-system $\systemP$ of $\pair$, an \emph{end} of the $\p 3$-decomposition is a component that meets only one member of $\systemP$. Suppose it admits $e$ ends $X_1,\dots, X_e$. Then by Lemma \ref{lm:end} $X_i\cap\HK$ is a union of an 
annulus and a once-punctured closed surface of genus $g_i$. 
If $\sum_{i=1}^e g_i=g$, we say $\systemP$ is \emph{semi-full}. A \emph{semi-fully} $\p 3$-decomposable $\pair$ is one that admits a semi-full $\p 3$-system. When $g=2$, every $\p 3$-system is semi-full; for $g>2$, see Figs.\ \ref{fig:pairwheel} and \ref{fig:wheelfive} for an example.  

\begin{theorem}[{Theorem \ref{teo:higher_genus_symmetry}}]\label{intro:teo:symmetry_higher}
Suppose $\pair$ is $2$-indecomposable, semi-fully $\p 3$-decomposable. 
If $\Compl\HK$ is $\partial$-irreducible, then for every finite subgroup $H<\Sym\HK$, the order $\vert H\vert$ divides $e$; if in addition $e=g$ and $\Compl\HK$ is atoroidal, then $\vert\Sym\HK\vert$ divides $g$. 
\end{theorem}

We apply Theorem \ref{intro:teo:symmetry_two} to determine the chirality of  
$\sixten$ in the handlebody knot table \cite{IshKisMorSuz:12} at the end of Subsection \ref{subsubsec:genus_two_classification}, which was previously unknown; see
Ishii-Iwakiri-Jang-Oshiro \cite{IshIwaJanOsh:13}.
Theorem \ref{intro:teo:symmetry_higher} also implies that, if $e=g$ is an odd prime, then $\pair$ is always chiral (Corollary \ref{cor:odd_prime_symmtery}).
\cout{
Lastly, in Section \ref{subsec:motto_hks}, Theorem \ref{intro:teo:uniqueness} is used to distinguish inequivalent handlebody-knots with homeomorphic exteriors, giving an alternative proof for the inequivalence of Motto's handlebody-knots \cite{Mott:90}.
}
\subsection*{Conventions}
We work in the piecewise linear category.
Given a subpolyhedron $X$ of a $3$-manifold $M$, $\overline{X}$, 
$\mathring{X}$, $\mathfrak{N}(X)$, and 
$\partial_f X$ denote the closure, the interior, a regular neighborhood, and 
the frontier of $X$ in $M$, respectively.
The \emph{exterior} $\Compl X$ of $X$ in $M$ is defined to be $\overline{M-\rnbhd{X}}$ 
if $X\subset M$ is of positive codimension, and to $\overline{M-X}$ otherwise. 
By $\vert X\vert$, we understand the number of components in $X$. 
 
A surface $S$ other than a disk in a $3$-manifold $M$ is \emph{essential} 
if it is incompressible and $\partial$-incompressible, while a disk $D\subset M$ is \emph{essential} if $D$ does not cut off from $M$ a $3$-ball. We denote by $(\sphere,X)$ an embedding of $X$ in the oriented $3$-sphere $\sphere$, and 
two embeddings $(\sphere,X_1)$, $(\sphere,X_2)$
are equivalent if there is an orientation-preserving self-homeomorphism of $\sphere$ carrying $X_1$ to $X_2$. We use $\simeq$ to denote equivalent embeddeings as well as homeomorphic spaces. 
Throughout the paper, $\pair$ 
is a genus $g$ handlebody-knot. 

Given two surfaces $P,Q$ in a $3$-manifold $M$,
we denote by $\alphaP$ and $\alphaQ$ the intersection $P\cap Q$ when viewing it as a submanifold in $P$ and $Q$, respectively.
Let $n_\partial$ (resp.\ $n_\circ$) be 
the number of arcs (resp.\ circles) in $P\cap Q$. Then the pair $\tind$ is called the \emph{total index} of $P,Q$.
If we consider only the number $n_\otimes$ of circles in $P\cap Q$ \emph{inessential} in $P$ or $Q$, then the pair $\pind$ is called the \emph{partial index} of $P,Q$.


\section{Unique decomposition theorem}\label{sec:uniqueness}
\subsection{Arc configuration}\label{subsec:arcs}
Let $P\subset\Compl\HK$ be a $\p 3$-surface  
and $Q\subset\Compl\HK$ a surface. 
Consider the set $\mathbb{N}_0\times \mathbb{N}_0$ of non-negative integer pairs, and the partial order $\prec$ on $\mathbb{N}_0\times \mathbb{N}_0$ defined as follows:
$(a,b)\prec  (c,d)$ if 
either $a<c$ or $a=c$ and $b<d$.  
We assume that $P,Q$ are so given that they minimize 
the partial index $\pind$ 
in their isotopy classes.

We denote by $C_1,C_2,C_3$ the components of 
$\partial P$ 
so that $C_2$
is in between $C_1,C_3$ in $\partial\HK$ (see Fig.\ \ref{fig:P_cuts_surface}).  
The circles $C_1,C_2,C_3$ cut $\partial\HK$ into two annuli and two once-punctured closed surfaces of positive genus. Denote by $A_{2i}$ the annulus bounded by $C_2,C_i$, and by $F_i$ the once-punctured closed surface bounded by $C_i$, $i=1$ or $3$. Let $X,Y\subset\Compl\HK$ be the components cut off by $P$. It may be assumed that $F_1,A_{23}\subset X$ and $F_3,A_{21}\subset Y$ (see Fig.\ \ref{fig:P_cuts_surface}). 
The disks bounded by $C_1,C_2,C_3$ cut $\HK$ into four components $V_1,V_{21},V_{23},V_{3}$ 
with $F_i\subset\partial V_i$, $A_{2i}\subset \partial V_{2i}$, $i=1,3$. 
\begin{figure}[t]
\begin{subfigure}{.51\linewidth}
\centering
\begin{overpic}[scale=.16,percent]{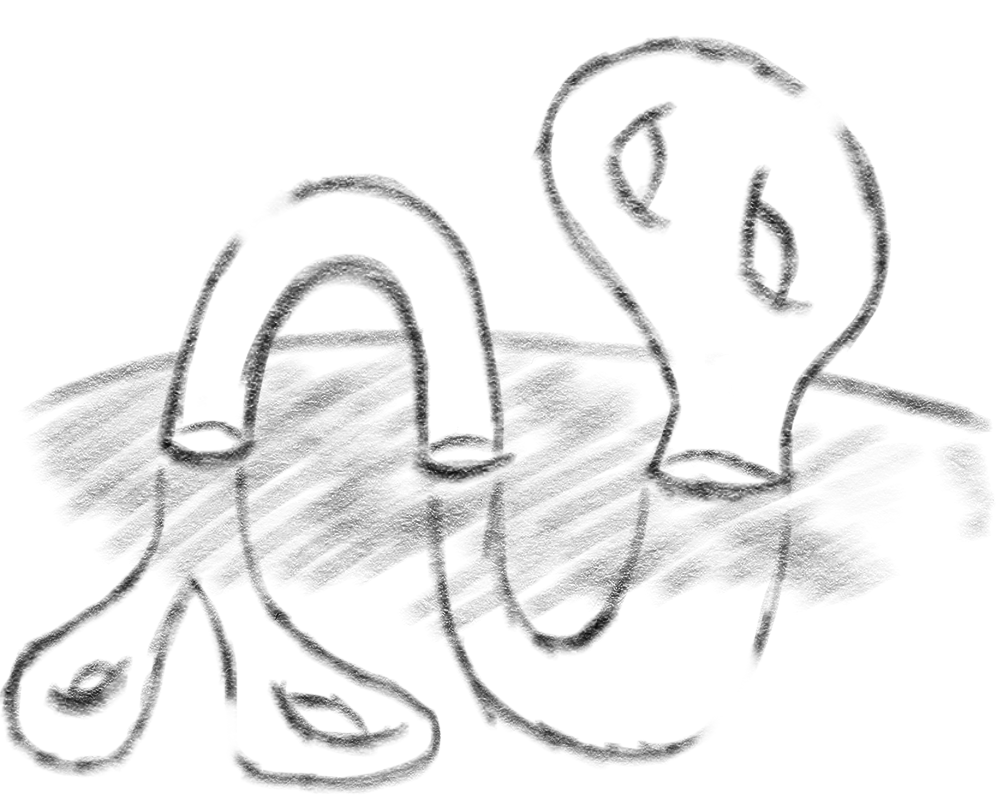}
\put(87,28){$P$}
\put(81,68){$F_1$}
\put(71,35.2){\footnotesize $C_1$}
\put(71,8){$A_{12}$}
\put(42.8,37.3){\footnotesize $C_2$}
\put(22,56.5){$A_{23}$}
\put(18,38.6){\footnotesize $C_3$}
\put(20.5,4.7){$F_3$}
\put(5,70){$X$}
\put(93,4){$Y$}
\end{overpic}
\caption{Cutting $\Compl\HK$ along $P$.}
\label{fig:P_cuts_surface}
\end{subfigure} 
\begin{subfigure}{.45\linewidth}
\centering
\begin{overpic}[scale=.15,percent]{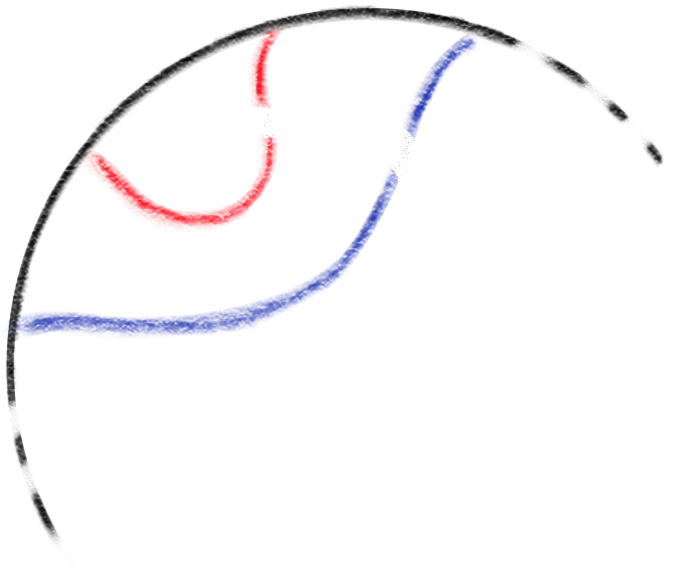}
\put(37,65.8){\footnotesize $\alpha$}
\put(55.7,61){\footnotesize $\beta$} 
\put(38,46){$D$}
\end{overpic}
\caption{$\alpha$, $\beta$ parallel in $Q$.}
\label{fig:alpha_beta_in_Q}
\end{subfigure}
\begin{subfigure}{.45\linewidth}
\centering
\begin{overpic}[scale=.15,percent]{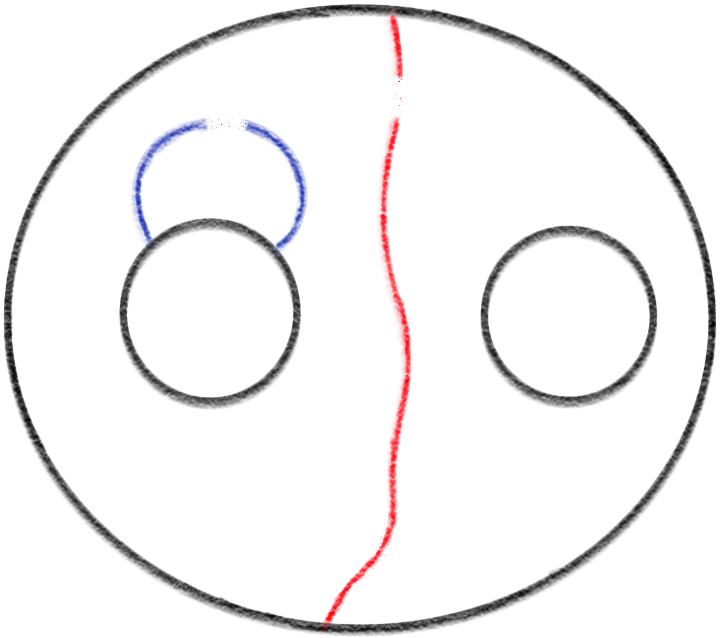}
\put(29,70){\footnotesize $\beta$}
\put(53.5,72.8){\footnotesize $\alpha$}
\put(27,61){\footnotesize $D'$}
\end{overpic}
\caption{$\alpha$, $\beta$, $D'\subset P$.}
\label{fig:alpha_beta_in_P}
\end{subfigure}
\begin{subfigure}{.5\linewidth}
\centering
\begin{overpic}[scale=.19,percent]{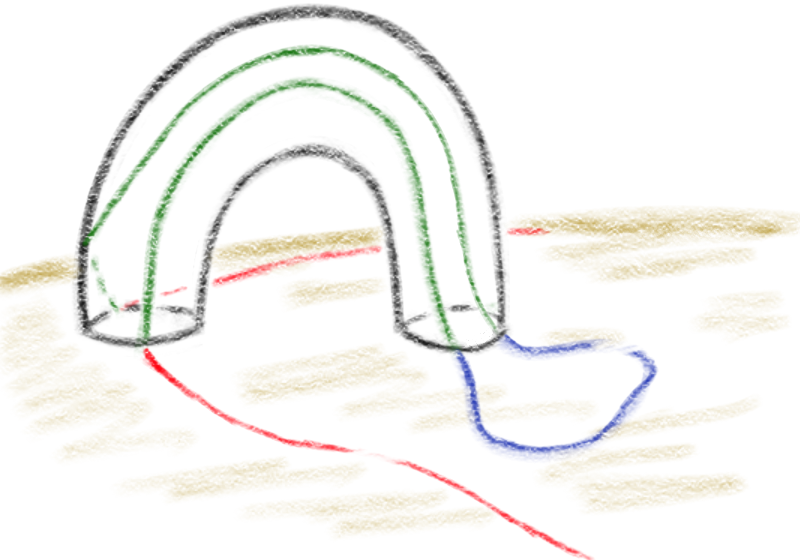}
\put(45,12.7){\small $\alpha$}
\put(75,26){\small $\beta$}
\put(64,18){\small $D'$}
\put(11.1,37.5){\small $D''$}
\put(10,10){\small $A$}
\end{overpic}
\caption{$D',A\subset P$ and $D''\subset A_{21}$ or $A_{23}$.}
\label{fig:Dprimeprime}
\end{subfigure}
\caption{}
\end{figure}

An arc $\alpha$ in $\alphaQ$ 
is called an 
\emph{$(\mathtt{i},\mathtt{j})$-arc} if $\partial\alpha$ 
meets the components $C_i,C_j\subset\partial P$.
An $(\mathtt{i},\mathtt{i})$-arc is abbreviated to an \emph{$\mathtt{i}$-arc} (resp.\ \emph{$\mathtt{i^\ast}$-arc}) if $\alpha$ is essential (resp.\ inessential) in $P$.  
An $(\mathtt{i},\mathtt{j})$-arc 
is always essential in $P$ 
when $i\neq j$.

Two arcs $\alpha_1,\alpha_2$ in $\alphaQ$ are \emph{immediately parallel} if they cut off a disk $D$ from $Q$ with $D\cap\alphaQ=\alpha_1\cup \alpha_2=\fron D$. An arc $\alpha$ 
in $\alphaQ$ is said to be \emph{outermost} if $\alpha$ cuts off from $Q$ a disk disjoint from arcs in $\alphaQ-\{\alpha\}$.
Let $\alpha_1,\alpha_2$ in $\alphaQ$ be two outermost arcs, and denote by $D_1,D_2$ the disks cut off by $\alpha_1,\alpha_2$, respectively. Then $\alpha_1,\alpha_2$ are \emph{boundary-adjacent} if $D_1,D_2$ are disjoint, and $\partial Q-(D_1\cup D_2)$ contains an arc disjoint from arcs in $\alphaQ-\{\alpha_1,\alpha_2\}$.

\begin{lemma}\label{lm:two_incompressible}
If $Q$ is incompressible, then $n_\otimes=0$.
\end{lemma} 
\begin{proof}
By the incompressibility of $P,Q$ and irreducibility of $\Compl\HK$.
\end{proof}

\begin{lemma}\label{lm:arcs}\hfill
\begin{enumerate}[label=\textnormal{(\roman*)}]
\item\label{itm:outermost} Suppose either $g=2$
or $\Compl\HK$ is $\partial$-irreducible. 
If $\alpha\in \alphaQ$ is an outermost arc, then
$\alpha$ is either a $\one$-arc
or a $\three$-arc.
\item\label{itm:one_three_exclusion} $\alphaQ$ cannot contain a $\one$-arc and $\three$-arc at the same time. 
\item\label{itm:no_adjacent_parallel} If $\alpha \in\alphaQ$ is an outermost $\one$-arc (resp.\ $\three$-arc), then $\alpha$ has no immediately parallel
arc and no boundary-adjacent arc. 
\item\label{itm:not_parallel_types} If $\alpha_1,\alpha_2\in \alphaQ$ 
are both $(\one,\two)$-arcs (resp.\ $(\two,\three)$-arcs), then they are not immediately parallel.
\end{enumerate}
\end{lemma}
\begin{proof}
\ref{itm:outermost} Note first by connectivity, $\alpha$ cannot be an $(\one,\three)$-arc.
Let $D\subset Q$ be the disk cut off by $\alpha$. If $\alpha$ is a $(\one,\two)$-arc (resp.\ $(\two,\three)$-arc), then $\alpha$
meets $C_2$ at one point. Thus the frontier of a regular neighborhood of $D\cup C_2$
in $Y$ (resp.\ $X$) 
induces a compressing disk of $P$,
contradicting the incompressibility of $P$. Let $A$ be $A_{21}$ or $A_{23}$. 
If $\alpha$ is a $\two$-arc, then 
$D\cap A\subset A$ is an inessential arc and cuts off from $A$ a disk $D'$. Isotoping $Q$ along $D'$ replaces the arc $\alpha$ with a circle, contradicting the minimality of $\pind$.  

Suppose $\alpha$ is a $\one^*$-arc (resp.\ $\three^\ast$-arc). Then $\alpha$ cuts off a disk $D'$ from $P$. By the minimality, the arc $D\cap F_1$ (resp.\ $D\cap F_3$) is essential 
in $F_1$ (resp.\ $F_3$), and hence the union $D\cup D'$ induces an essential disk $D''$ in $X$ (resp.\ $Y$). 
This cannot happen 
if $\Compl\HK$ is $\partial$-irreducible. If $g=2$, then $V_1$ (resp.\ $V_3$) is
a solid torus, and hence $D''$ meets an essential disk $D_v\subset V_1$ disjoint from $C_1$ (resp.\ $D_v\subset V_3$ disjoint from $C_3$) at one point. The frontier of a regular neighborhood of $D''\cup \partial D_v$ then induces a compressing disk of $P$, a contradiction.

\ref{itm:one_three_exclusion} If $\alphaQ$ contains a $\one$-arc (resp.\ $\three$-arc) $\alpha$, then $\alpha$, being essential in $P$, separates $C_2$ and $C_3$ (resp.\ $C_2$ and $C_1$). Hence, every arc $\beta\subset \alphaP$ with boundary 
in $C_3$ (resp.\ $C_1$) is inessential. 

\ref{itm:no_adjacent_parallel} 
Suppose $\beta$ is an outermost arc boundary-adjacent to 
$\alpha$, and $D\subset Q$ the disk cut off by $\beta$ with $D\cap\alpha=\emptyset$. 
Since $\beta$ is a $(\two,\three)$-arc (resp.\ $(\one,\two)$-arc), $D\cap C_2$ 
is a point and thus the frontier of a regular neighborhood of $D\cup C_2$ in $X$ (resp.\ $Y$) induces a compressing disk of $P$, a contradiction.

Suppose $\beta$ is immediately parallel to $\alpha$. 
Then since $C_1,C_2,C_3$ are parallel and 
$\alpha$ is outermost in $Q$ and essential in $P$, $\beta$ is a $\two^\ast$-arc. Let $D\subset Q$ be the disk cobound by $\alpha$ and $\beta$ (see Fig.\ \ref{fig:alpha_beta_in_Q}), and 
$D'$ be the disk cut off by $\beta$ from $P$ (see Fig.\ \ref{fig:alpha_beta_in_P}). Observe that the closure of $\partial D-\alpha\cup\beta\subset Q$ are two parallel arcs in $A_{21}$ (resp.\ $A_{23}$), which cut $A_{12}$ (resp.\ $A_{23}$) 
into two disks, one of which, 
denoted by $D''$, meets $D'$ in an arc (see Fig.\ \ref{fig:Dprimeprime}).
Since $\alpha$ is essential, $\alpha$ 
cuts $P$ into two annuli, one of which, denoted by $A$, meets $D''$. 
If $A$ contains $D'$, then isotoping $D$ along $D'\cup D''$ induces a disk bounded by a loop parallel to $C_2$, an impossibility. 
If $A$ does not contain $D'$, 
then isotoping $D$ along $D'\cup D''$ induces a disk bounded by a loop parallel to $C_3$ (resp.\ $C_1$),  
contradicting the incompressibility of $P$.  

\ref{itm:not_parallel_types} Let $D\subset Q$ be the disk cut off by 
$\alpha_1\cup\alpha_2$ disjoint from 
other components in $\alphaQ$. 
If $D$ is in $X$ (resp.\ $Y$), 
then the arc $D\cap A_{21}$ (resp.\ $D\cap A_{23}$) is inessential. Isotope through the disk in $A_{21}$ (resp.\ $A_{23}$) cut off by the arc reduces  $\vert \partial P\cap \partial Q\vert$, a contradiction. If $D$ is in $Y$ (resp.\ in $X$), then, by the minimality, $D\cap A_{21}$ (resp.\ $A_{23}$) consists of two parallel essential arcs, contradicting 
the fact that no knot exterior admits a disk with non-integral 
boundary slope.
\end{proof}

\begin{lemma}\label{lm:two_p3_surfaces}
Suppose either $g=2$ or $\Compl\HK$ is $\partial$-irreducible. If $Q$ is a $\p 3$-surface, then $\pind=(0,0)$.
\end{lemma}
\begin{proof}  
Denote by $C_1',C_2',C_3'$ components of 
$\partial Q$, and observe that, since components of $\partial P$ are parallel, by the minimality of $n_\partial$, every component of $\partial Q$ meets $\partial P$ at the same number $n$ of points. 
Since every $C_i$ (resp.\ $C_i'$) separates $\partial\HK$, $n=12k$ with $k$ a non-negative integer. Also, we have $3n=2n_\partial$. Since $C_2$ is in between $C_1,C_3$ in $\partial\HK$,
the labels for points in $C_i'\cap\partial P$ follows the pattern around $C_i'$:  
\begin{equation}\label{eq:pattern}
1,2,3,3,2,1,\dots,1,2,3,3,2,1.
\end{equation}

{\bf Claim: No separating essential arc in $\alphaQ$ and $\alphaP$.}
Consider first $\alphaQ$ and suppose otherwise. Then it may be assumed there exists an arc $\alpha$ separating 
$Q$ into two annuli with $C_2',C_3'$ in different annuli. In particular, 
every essential arc meeting $C_i'$, $i=2,3$,
is non-separating and meets $C_1'$.
Let $n_i$ be the number of essential arcs meeting $C_i'$, $i=1,2,3$.
By Lemma \ref{lm:arcs}\ref{itm:outermost}\ref{itm:one_three_exclusion}, it may be assumed outermost inessential arcs in $\alphaQ$ are all $\one$-arcs. Since $n_\otimes=0$ by Lemma \ref{lm:two_incompressible}, 
all inessential arcs in $\alphaQ$ are outermost 
by Lemma \ref{lm:arcs}\ref{itm:no_adjacent_parallel}.
In addition, since every $(\mathtt{j},\mathtt{k})$-arc in $\alphaQ$ with $\mathtt{j}=2$ or $3$ is essential, we have $n_i\geq \frac{2}{3}n$, $i=2,3$, and hence $n_1=n_2+n_3>n$, a contradiction. 
Swapping the roles of $P$ and $Q$ proves the claim for $\alphaP$.  

The fact that no separating essential arc exists in $\alphaP$ implies no $\one$-arc of $\three$-arc in 
$\alphaQ$. Thus, all arcs in $\alphaQ$ are essential and non-separating.

Suppose now $n\neq 0$, and hence $n=12k\geq 12$. Then there are at least $6$ essential arcs $\alpha_i$, $i=1,\dots, 6$, connecting $C_1',C_2'$. 
Let $a_i,b_i$ be the boundary of $\alpha_i$ in $C_1',C_2'$, respectively.
Because one period of the pattern \eqref{eq:pattern} is $(1,2,3,3,2,1)$, 
it may be assumed that $(a_1,\dots,a_6)$ 
are labeled with $(1,2,3,3,2,1)$; see Fig.\ \ref{fig:pattern}. 
There are six possible labeling patterns on $(b_1,\dots,b_6)$:
\begin{multline}\label{eq:possible_patterns}
(1,2,3,3,2,1),\hspace*{.2em}
(2,3,3,2,1,1),\hspace*{.2em} 
(3,3,2,1,1,2),\hspace*{.2em}\\
(3,2,1,1,2,3),\hspace*{.2em}
(2,1,1,2,3,3),\hspace*{.2em}
(1,1,2,3,3,2).  
\end{multline}
The first, second, fourth and sixth cases in \eqref{eq:possible_patterns} imply $\alphaQ$ contains an $i$-arc or $i^\ast$-arc, $i=1,2$ or $3$, contradicting $\alphaP$ containing no separating arcs. The third and fifth cases both imply the existence of a pair of successive $(\one,\two)$-arcs 
and a pair of successive $(\two,\three)$-arcs, at least one of which is immediately parallel, contradicting Lemma \ref{lm:arcs}\ref{itm:not_parallel_types}; see Fig.\ \ref{fig:pattern}.  
\end{proof}

\begin{figure}[t]
\begin{subfigure}[b]{.51\linewidth}
\centering
\begin{overpic}[scale=.18,percent]{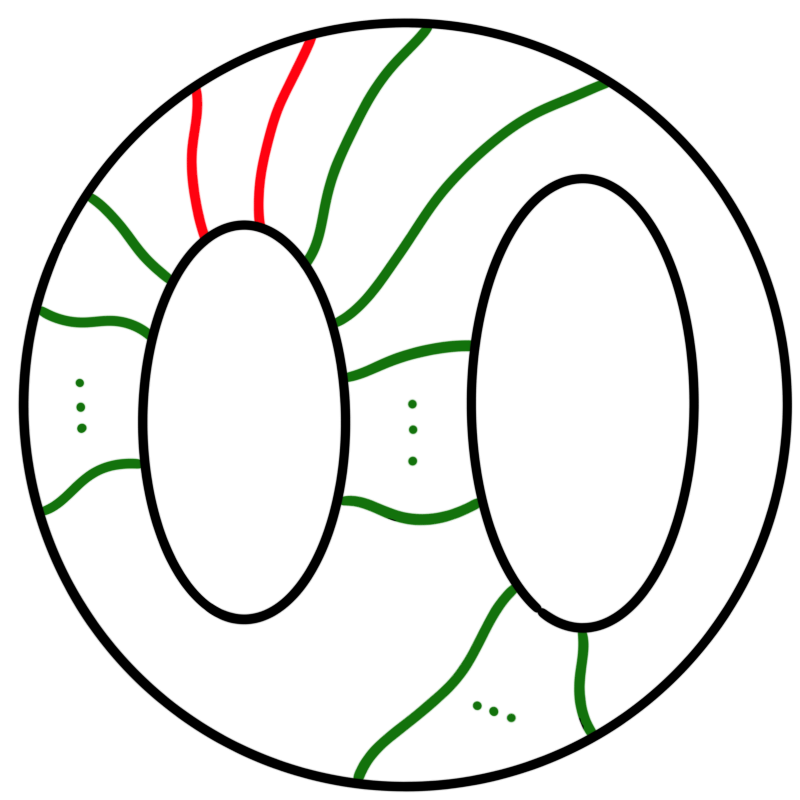}
\put(.5,61){\small $3$}
\put(7,76){\small $3$}
\put(21.5,90.5){\small $2$}
\put(36,96.5){\small $1$}
\put(52.5,98){\small $1$}
\put(76,91){\small $2$}
\put(20,55){\footnotesize $2$}
\put(22.1,62){\footnotesize $1$}
\put(25,65){\footnotesize $1$}
\put(31,66.2){\footnotesize $2$}
\put(35.5,62.4){\footnotesize $3$}
\put(38,56){\footnotesize $3$}
\put(8,55.5){\footnotesize $\alpha_3$}
\put(11,66){\footnotesize $\alpha_4$}
\put(17,75){\footnotesize $\alpha_5$}
\put(34.2,81){\footnotesize $\alpha_6$}
\put(44.6,82){\footnotesize $\alpha_1$}
\put(60,78){\footnotesize $\alpha_2$}
\put(89,79){$C_1'$}
\put(25.3,25.5){$C_2'$}
\end{overpic}
\caption{Labeling pattern.}
\label{fig:pattern}
\end{subfigure} 
\begin{subfigure}[b]{.47\linewidth}
\centering
\begin{overpic}[scale=.15,percent]{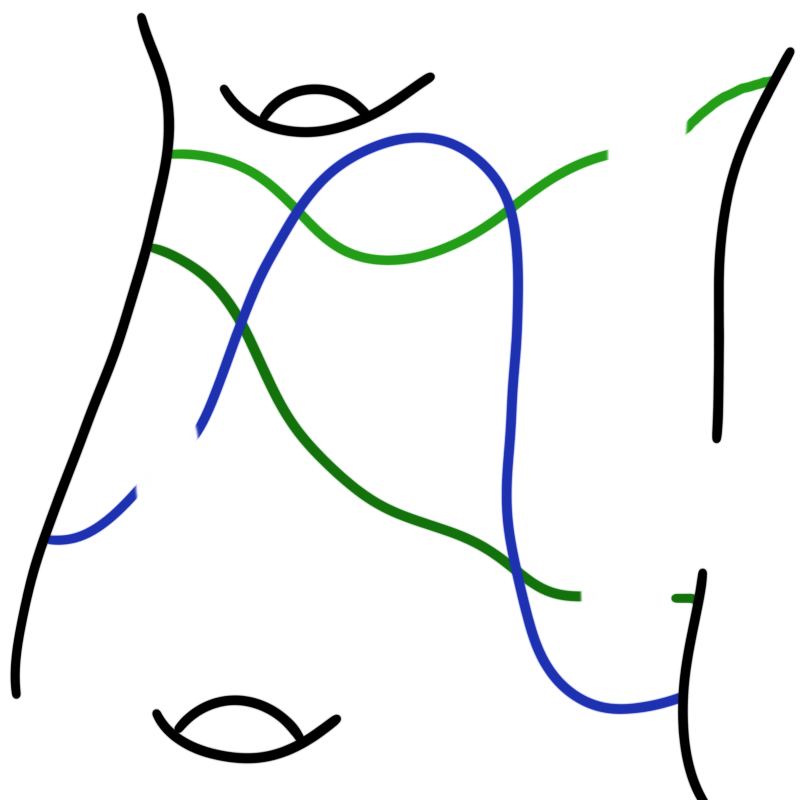}
\put(17,39.5){\small $\partial P$}
\put(72.4,22.3){\small $\partial Q$}
\put(74.2,81.6){\small $\partial Q'$}
\put(40,34.7){\footnotesize $\gamma$}
\put(65,45){\footnotesize $\beta$}
\put(50,62.5){\footnotesize $\gamma'$}
\put(50,50){\small $D$}
\put(47,73){\small $D'$}
\put(82,34){\small $\partial\HK$}
\end{overpic}
\caption{Bigon criterion.}
\label{fig:bigon}
\end{subfigure}
\caption{}
\end{figure}

\subsection{Uniqueness of decomposition}
Let $\systemP$ be a $\p 3$-system of $\pair$ and $\uniP$ the union of members in $\systemP$.

\begin{lemma}\label{lm:mini_index}
Let $\uniQ$ be an incompressible surface. If
$\uniP,\uniQ$ minimize the total index $(\bn_\partial,\bn_\circ)$ in their isotopy classes, then any two components $P\subset \uniP, Q\subset\uniQ$ minimize the partial index $\pind$ in their isotopy classes.   
\end{lemma}
\begin{proof}
Suppose $\pind$ is not minimized, and there are 
$\tilde P, \tilde Q$ isotopic to $P,Q$ respectively, 
with $\pindtilde\prec \pind$. 
In particular, either $\tilde n_\partial <n_\partial$ or 
$\tilde n_\partial=n_\partial$ and $\tilde n_\otimes <n_\otimes$.
In the former, 
by \cite[Proposition $1.7$]{FarMar:11}, $\partial P,\partial Q$ cut off a disk $D$ from $\partial\HK$ with $\partial D=\beta\cup\gamma$ and $\beta,\gamma$ two arcs in $\partial P, \partial Q$, respectively. If  
$(\uniQ-Q)\cap D=\emptyset$, then one can isotope $Q$ via $D$ so that members in $\systemQ$ remain mutually disjoint and 
$\partial \gamma$ is removed from $\partial\uniP\cap\partial\uniQ$ with no new intersection introduced, contradicting the minimality of
$(\bn_\partial,\bn_\circ)$. 
If $(\uniQ-Q)\cap D\neq \emptyset$, 
then there exists an outermost arc $\gamma'\subset \uniQ\cap D$ in $D$ that cuts off a disk $D'\subset D$ disjoint from $\gamma$; see Fig.\ \ref{fig:bigon}. Let $Q'\subset \uniQ$ be the component containing $\gamma'$. Then one can isotope $Q'$ across $D'$ so that the intersection $\partial\gamma'\subset \partial\uniP\cap\partial\uniQ$ is removed while members in $\systemQ$ remains mutually disjoint and no new intersection is created, 
contradicting $(\bn_\partial,\bn_\circ)$ is minimized.

Now, if $\tilde n_\partial =n_\partial$, then $0\leq\tilde n_\otimes <n_\otimes$, and hence there exists an innermost inessential circle $\alpha\in\alphaP$; $\alpha$ is inessential in $Q$ by incompressibility of $\uniQ$. Let $D_p,D_q$ be the disks cut off by 
$\alpha$ from $P,Q$, respectively. 
Then the union $D_p\cup D_q$ bounds a $3$-ball $B$ in $\Compl\HK$. If $(\uniQ-Q)\cap D_p=\emptyset$, then one can isotope 
$Q$ across 
$B$ so that the intersection $\alpha$ is removed and 
$\systemQ$ remains a set of disjoint surfaces with no new intersection introduced,
contradicting the minimality 
of  
$(\bn_\partial,\bn_\circ)$.
If $(\uniQ-Q)\cap D_p\neq\emptyset$, 
then there is an innermost circle $\beta\subset (\uniQ-Q)\cap D_p$ in 
$D_p$ that cuts off a disk $D_p'$ from $D_p$. Let $Q'$ be the component containing $\beta$, and $D_q'\subset Q'$ the disk cut off by $\beta$.  
Then 
$D_p'\cup D_q'$ bounds a $3$-ball $B$. Since $\mathring{D}_p'\cap \uniQ=\emptyset$, one can isotope $Q'$ across $B'$ so that the intersection $\beta$ is removed without creating new intersection or altering disjointness of $\systemQ$, contradicting 
the minimality of $(\bn_\partial,\bn_\circ)$.
\end{proof}

Let $\systemQ$ be another $\p 3$-system 
of $\pair$, and $\uniQ$ the union of its members.
\begin{lemma}\label{lm:two_systems}
Suppose $\pair$ is $2$-indecomposable and 
either $g=2$ or  
$\Compl\HK$ is $\partial$-irreducible. 
Then there is an ambient isotopy $f_t:\pair\rightarrow \pair$ such that
$f_1(\uniP)$ is disjoint from 
$\uniQ$.  
\cout{
Suppose $\pair$ is $2$-indecomposable and 
either $g=2$ or $g>2$ and $\Compl\HK$ is $\partial$-irreducible. Let $\systemP=\{P_1,\cdots,P_k\},\systemQ=\{Q_1,\cdots, Q_l\}$ be two $\p 3$-systems; denote by $\uniP,\uniQ$ the unions $\cup_{i=1}^k\systemP,  
\cup_{i=1}^l\systemQ$, respectively. 
Then there exists an ambient isotopy $f_t:\pair\rightarrow \pair$ such that
$f_1(\mathcal{P})$ is disjoint from 
$\mathcal{Q}$.
}
\end{lemma}
\begin{proof}
By the isotopy extension theorem, it suffices to show that $\uniP, \uniQ$ can be isotoped so that they are disjoint.
Isotope $\uniP,\uniQ$ so that they minimize the total index $\bdtind$. 
By Lemma \ref{lm:mini_index}, for every two components $P\in \systemP$, $Q\in \systemQ$, $P,Q$ minimize the partial index $\pind$ in their isotopy classes. Thus by 
Lemma \ref{lm:two_p3_surfaces}, 
we have $\pind=(0,0)$, and hence $\bn_\partial=0$.

Suppose $\bn_\circ\neq 0$. Then there exists 
a circle $\alpha$ in $\uniP\cap\uniQ$  
essential in both $\uniP,\uniQ$. Let 
$P\in\systemP, Q\in\systemQ$ be the components containing $\alpha$. It may be
assumed $\alpha$ is outermost in $P$ and cuts off an annulus $A_p$ disjoint from 
$\uniQ-\alpha$. Denote by $A_q$ the annulus cut off by $\alpha$ from $Q$. 
Then $A:=A_p\cup A_q$ is an incompressible annulus with $\partial A$ bounding two disks in $\HK$. Because $\pair$ is $2$-indecomposable, $A$ is $\partial$-parallel and hence $A_p\cup A_q$ cuts off a solid torus $V$ from $\Compl \HK$ with $H_1(A_q)\rightarrow H_1(V)$ an isomorphism. Thus $Q$ can be isotoped through $V$ in $\Compl\HK-(\uniQ-Q)$ so that the intersection $\alpha$ is removed, contradicting the minimality.
Therefore $\bn_\circ=0$ and hence the assertion.
\end{proof}
 
\begin{theorem}\label{teo:uniqueness}
Suppose $\pair$ is $2$-indecomposable, and either $g=2$ or $\Compl\HK$ is $\partial$-irreducible. 
If $\systemP,\systemQ$ both are maximal,  
then there exists an ambient isotopy $f_t:\pair\rightarrow \pair$ such that
$f_1(\uniP)=\uniQ$.  
\end{theorem} 
\begin{proof}
It follows from Lemma \ref{lm:two_systems} 
and the maximality of $\systemP$ and $\systemQ$. 
\end{proof}

\section{Topology of $\p 3$-decomposable handlebody-knots}\label{sec:topology}
Let $\pair$ be a
$\p 3$-decomposable handlebody-knot, 
$\systemP$ a $\p 3$-system, and $\uniP$ 
the union of 
the $\p 3$-surfaces in $\systemP$.

\subsection{Components of $\p 3$-decomposition}  
A component $X\subset\Compl\HK$ cut off by $\uniP$ is called  
an \emph{end} if $X$ meets $\uniP$ at one component, namely $\vert\partial_f X\vert=1$. 

\begin{lemma}\label{lm:end}
Given a component $X\subset\Compl\HK$ cut off by $\uniP$, the following three statements are equivalent:
\begin{enumerate}[label=\textnormal{(\roman*)}]
\item\label{itm:end} $X$ is an end.  
\item\label{itm:union} $X\cap \HK$ is a union of a once-punctured closed surface 
and an annulus. 
\item\label{itm:one_torus} $X\cap \HK$ contains a once-punctured closed surface $F$. 
\end{enumerate}
\end{lemma}
\begin{proof}
\ref{itm:outermost}$\Rightarrow$\ref{itm:union}: Since $X$ is an end, the frontier $P:=\partial_f X$ has one component.
Since components of $\partial P$ are parallel in $\partial\HK$, both $X\cap\HK$ and $\Compl X\cap\HK$ consist of a once-punctured closed surface and an annulus.  

\ref{itm:union}$\Rightarrow$\ref{itm:one_torus} is clear.
To see \ref{itm:one_torus}$\Rightarrow$\ref{itm:outermost}, we suppose $X$ is not an end, and $P$ is the component of the frontier $\partial_f X$ that meets the once-punctured closed surface $F$, and $P'\subset\partial_f X$ is another component.
Let $Y,Y'\subset \Compl\HK$ be the components cut off by $P,P'$ with $X\cap Y=P,X\cap Y'=P'$, respectively. 
Then $\Compl Y\cap\HK$ consists of a once-punctured closed surface $F'$ and an annulus $A$. By $F\subset X\subset\Compl Y$, we have $F=F'$.  
Since $P'\cap F=\emptyset$, the connectedness of $\HK$ implies that $Y'\cap \partial\HK=Y'\cap A$ is a union of annuli, 
contradicting $P'$ is a $3$-decomposing surface.  
\end{proof}
 
Lemma \ref{lm:end} implies that if $X$ is not an end, then every component of $X\cap\HK$ is an $n$-punctured closed surface $S$ with $n\geq 2$. The topology of $S$ is in relation to $\uniP$ is heavily constrained when $S$ is not an annulus.

\begin{figure}[h!]
\begin{subfigure}{.48\linewidth}
\centering
\begin{overpic}[scale=.17,percent]{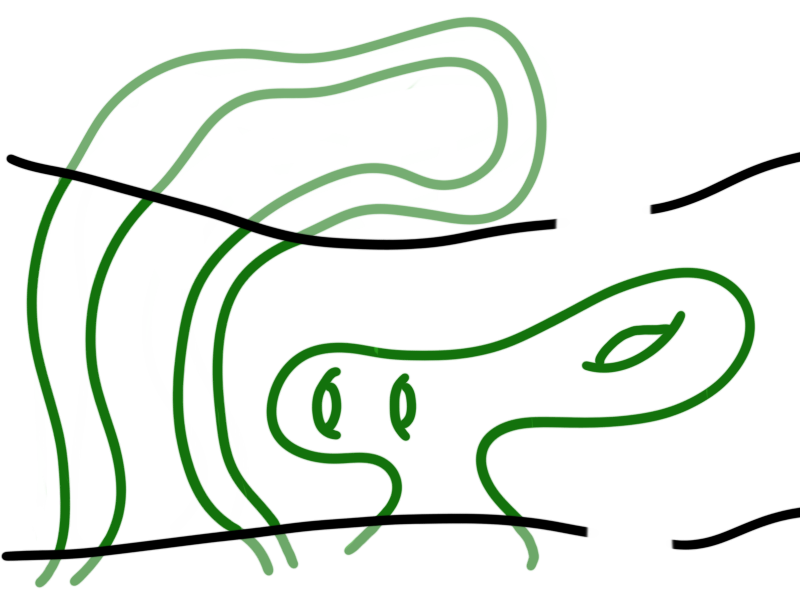}
\put(72,46){\small $P'$}
\put(75,5){\small $P$}
\put(92,18){\small $X$}
\put(93,0){\small $Y$}
\put(90,60){\small $Y'$}
\put(60,24){\small $F$}
\put(4.8,38){\small $A$}
\end{overpic}
\caption{Components $X,Y,Y'$.}
\label{fig:punctured_surface_end}
\end{subfigure}
\begin{subfigure}{.48\linewidth}
\centering
\begin{overpic}[scale=.17,percent]{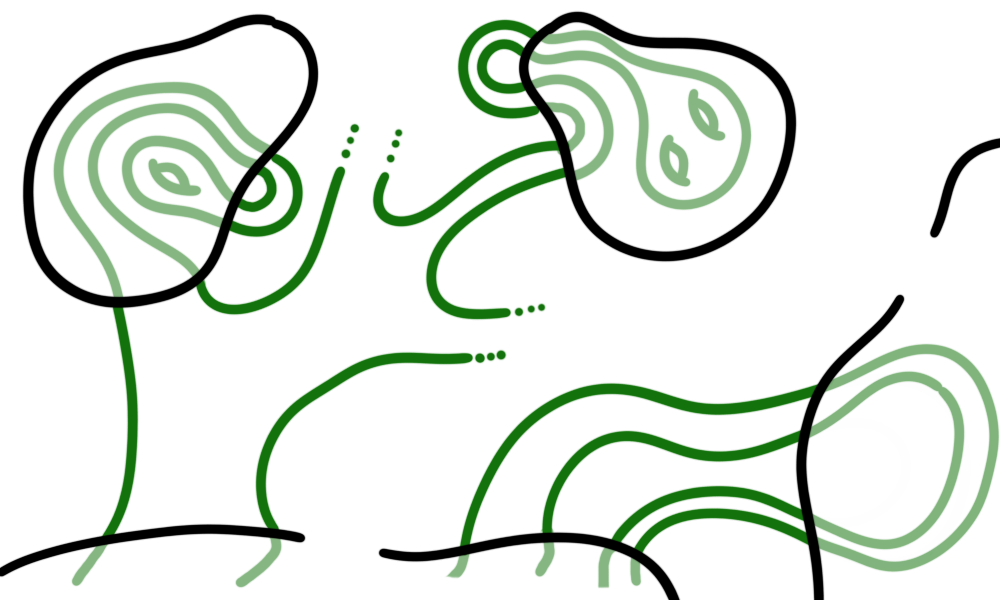}
\put(17,17){\small $S$}
\put(49,9){\small $A$}
\put(70,25){\small $X$}
\put(15,0){\small $Y$}
\put(85,12){\small $Y'$}
\put(32,4){\small $P$}
\put(90,31){\small $P'$}
\put(33,35){\footnotesize $F$}
\put(12,31){\footnotesize \texttransparent{0.5}{$F$}}
\put(64.7,48.5){\footnotesize \texttransparent{0.5}{$F$}}
\end{overpic}
\caption{Surfaces $S\subset F$ and $A$.}
\label{fig:punctured_sphere_meet_all}
\end{subfigure}
\caption{} 
\end{figure}

\begin{lemma}\label{lm:punctured_sphere}
Given a component $X\subset\Compl\HK$ cut off by $\uniP$, suppose $X\cap \HK$ contains an $n$-punctured closed surface $S$ with $n\geq 2$ and $S$ not an annulus. Then
\begin{enumerate}[label=\textnormal{(\roman*)}]
\item\label{itm:meetonce}
every component $P$ of $\partial_f X$ meets $S$ at exactly one component of $\partial P$;
\item\label{itm:otherannuli}
every other component in $X\cap\HK$ 
is an annulus $A$ with $A$ meeting only one component of $\partial_f X$.
\end{enumerate}
\end{lemma}
\begin{proof} 
Observe first there is at least one component $P\subset\uniP$ meeting $S$. Denote by $Y\subset \Compl\HK$ the component cut off by 
$P$ with $Y\cap X=P$, and by 
$F,A$ the once-punctured closed surface and annulus in $\Compl Y\cap \HK$, respectively. 
Since $S$ is not an annulus, we have $S\subset F$.

\ref{itm:meetonce}: 
Suppose there exists a component $P'$ 
disjoint from $S$. Let $Y'\subset\Compl\HK$ be the components cut off by $P'$ with $Y'\cap X=P'$. 
By the connectedness of $\partial\HK$, $P'$ meets $A\cup F$, yet because $S\cap P'=\emptyset$, we have $F\cap P'=\emptyset$ as well. Therefore, $Y'\cap\HK=Y'\cap A$, which is a union of annuli, an impossibility.
In addition, since $S\subset F$, we have $P\cap S= P\cap F=\partial F$, which is a circle. 

\ref{itm:otherannuli}: Since every component of $\partial_f X$ meets $S$, it suffices to show that every other component $S'$ of $X\cap\HK$ meeting $P$ is an annulus, and it is clear since $S'\subset A$. 
Suppose now $S'$ meets the other component $P'$ of $\partial_f X$. 
Then the facts that $S'$ being an annulus and $\partial P$ being parallel in $\partial\HK$ imply the two loops $S'\cap P', S\cap P$
are parallel in $\partial\HK$. Since $S$ 
is not an annulus, the intersection $S\cap P'$ cannot be parallel to $S\cap P$ and hence to $S'\cap P'$ in $\partial\HK$, contradicting $P'$ is a $\p 3$-surface.
Thus $S'$ meets $P$ only. 
\end{proof}
 
We have the following characterization of a $\p 3$-decomposable handlebody-knot.
\begin{theorem}\label{teo:ends_connectors}
Let $X$ be a component cut off from $\Compl\HK$. Then
\begin{enumerate}[label=\textnormal{(\roman*)}]
\item\label{itm:end_ii}
$\vert\partial_fX\vert=1$ 
if and only if $X\cap\HK$ is a union of 
once-punctured closed surface and an annulus.
\item\label{itm:higherconnector}
$\vert\partial_fX\vert\geq 2$ 
if and only if either $X\cap\HK$ is a union of 
three annuli with each meetings both components of $\partial_f X$ or $X\cap\HK$ is a union of an $n$-punctured closed surface and $n$ annuli, each of which meets only one component of $\partial_f X$.
\end{enumerate}
\end{theorem}
\begin{proof}
\ref{itm:end_ii} follows from Lemma \ref{lm:end}.
%
%
%
For \ref{itm:higherconnector}, the implication ``$\Leftarrow$" is clear.
To see the implication ``$\Rightarrow$", we consider first the case 
where $X\cap\HK$ only contains annuli. 
Then observe that $\partial_f X$ 
cuts $\Compl\HK$ into $n+1$ components: 
$X, Y_1,\dots, Y_n$, $n\geq 2$. Since $\HK$ is connected, 
the union $(X\cap\HK)\cup Y_1\cup\dots\cup Y_n$ is connected. Thus, $X\cap\HK$ being a union of annuli implies boundary components of $\partial_f X$ are mutually parallel in $\partial\HK$. In particular, $\partial_f X$ cuts $\partial\HK$ into two once-punctured closed surfaces $F_1,F_2$ and some annuli. This implies $n=2$. 

Suppose now $X\cap\HK$ contains a non-annular component $S$. Then $S$ is a $m$-punctured closed surface with $m\geq 2$ by \ref{itm:end_ii}.
Since $S$ is not an annulus, by Lemma \ref{lm:punctured_sphere}\ref{itm:meetonce}, $n=m$, and the rest of the assertion follows from Lemma \ref{lm:punctured_sphere}\ref{itm:otherannuli}. 
\end{proof}
If the $\p 3$-decomposition has $e$ ends
$X_1,\dots, X_e$, and $X_i\cap\HK$ contains a 
once-punctured closed surface of genus $g_i$,
then by Theorem \ref{teo:ends_connectors}, 
the intersection $\Compl {X_1\cup\dots\cup X_e}\cap \HK$ 
is a union of $e$ annuli and an $e$-punctured closed surface of genus $g_0$, and hence
$\sum_{i=1}^eg_i+g_0=g.$ 
The system $\systemP$ is said to be \emph{full} if $e=g$, and to be \emph{semi-full} if $g_0=0$. 

 
\subsubsection*{Induced handlebody-knots}
Let $X\subset\Compl\HK$ a component cut off by $\uniP$. Then $X$ induces two handlebody-knots $\inducedhkX$ and $\inducedhkx$, where $\hkX:=\Compl X$ and $\hkx:=\HK\cup X$ (see Fig.\ \ref{fig:inducedhk}). 
\cout{
If $X$ is not an end, 
then by Theorem \ref{teo:ends_connectors}, $\hkX$ has a canonical spine given by the dual of the disks in $\hkX$ bounded by boundary components of $\partial_f X$. On the other hand, if 
the $\p 3$-decomposition has $e$ ends and 
$e=g$, then for every end $X$, $X\cap\HK$
consists of an annulus and a once-punctured torus. Thus, the disks in $\hkX$ bounded $\partial_f X$ induces a canonical spine of $\hkX$.
}
If, in addition, $\systemP$ is semi-full and $X$ is not an end or $\systemP$ is full, then the disks in $\hkX$ bounded by $\partial_f X$ induces a spine $\sgX$ of $\hkX$ and hence a spatial graph $\inducedsgX$.
%
\cout{We denote by 
$\sgX$ the canonical spine given above and by $\inducedsgX$ the corresponding spatial graph.}   
\subsubsection*{Example}
\begin{figure}[b]
\begin{subfigure}[b]{.29\linewidth}
\centering
\begin{overpic}[scale=.145,percent]{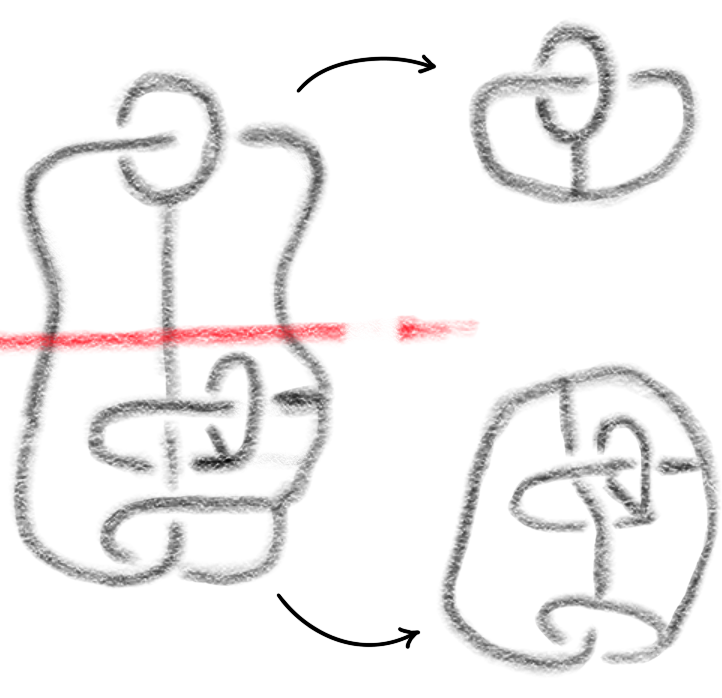}
\put(49,78){\tiny $\HK^X$}
\put(45,12){\tiny $\HK_X$}
\put(48,48){$P$}
\put(50,36){$X$}
\end{overpic}
\caption{$\HK_X$ and $\HK^X$.}
\label{fig:inducedhk}
\end{subfigure}
\begin{subfigure}[b]{.40\linewidth}
\centering
\begin{overpic}[scale=.18,percent]{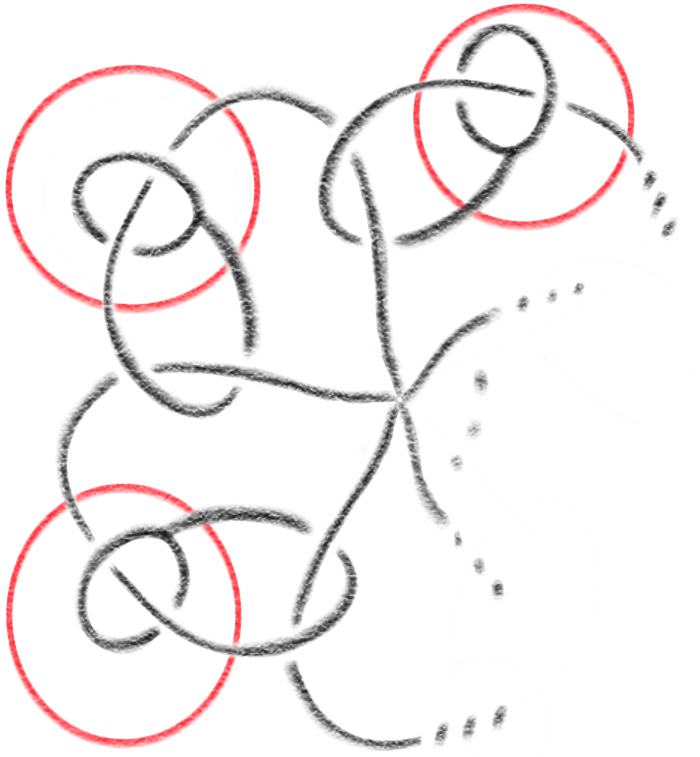}
\put(16,92){\footnotesize $P_1$}
\put(15,36.5){\footnotesize $P_2$}
\put(51.5,95){\footnotesize $P_g$}
\put(15,84){\tiny $X_{\! 1}$}
\put(4,25){\tiny $X_{2}$}
\put(75,90){\tiny $X_{\!g}$}
\put(42,52){\small $Y$}
\end{overpic}
\caption{$\pairwheel$ and $\systemP$.}
\label{fig:pairwheel}
\end{subfigure}
\begin{subfigure}[b]{.29\linewidth}
\centering
\begin{overpic}[scale=.14,percent]{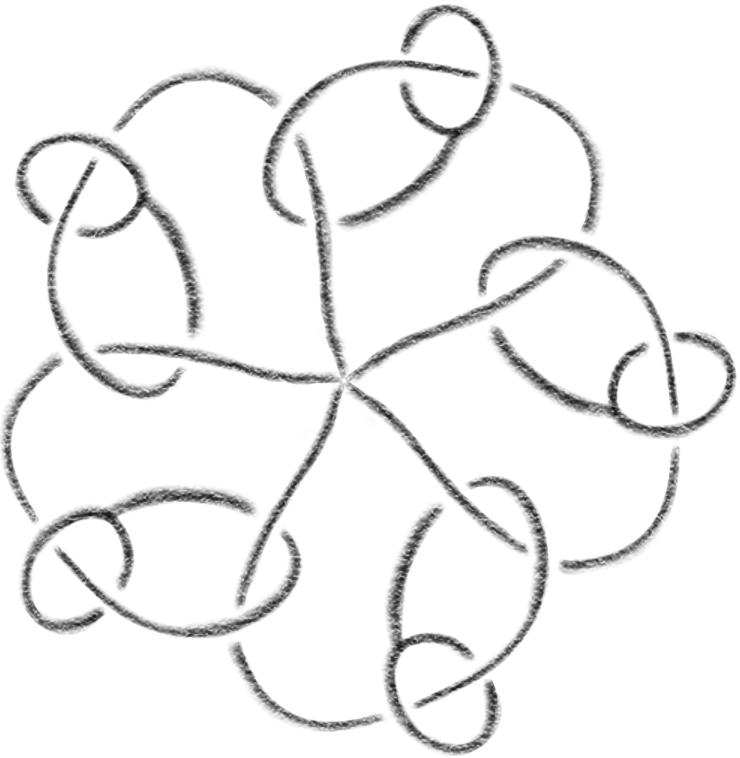}
\end{overpic}
\caption{$(\sphere,\mathrm{W}^5)$.}
\label{fig:wheelfive}
\end{subfigure} 
\begin{subfigure}[b]{.33\linewidth}
\centering
\begin{overpic}[scale=.15,percent]{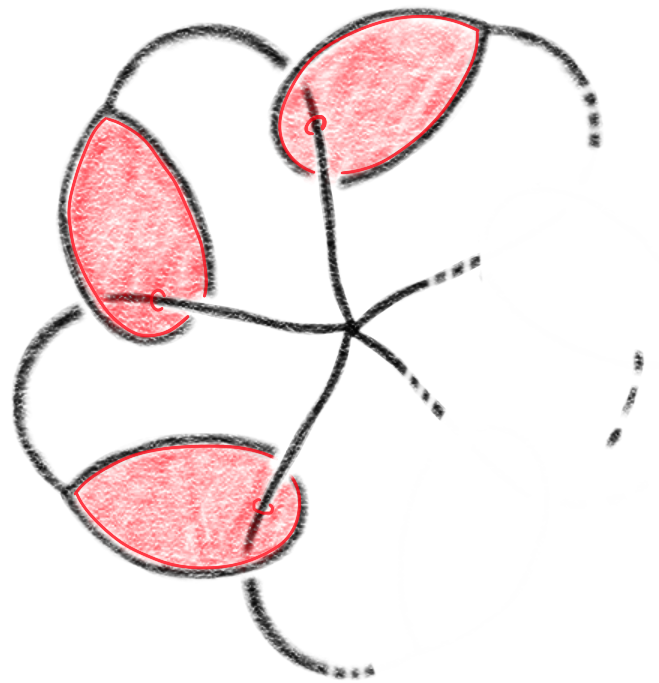}
\put(53,86){\small $A_g$}
\put(16,64){\small $A_1$}
\put(24,24){\small $A_2$}
\end{overpic}
\caption{Annulus $A_i\subset\Compl {\wheel_Y}$.}
\label{fig:wheelY}
\end{subfigure}
\begin{subfigure}[b]{.33\linewidth}
\centering
\begin{overpic}[scale=.14,percent]{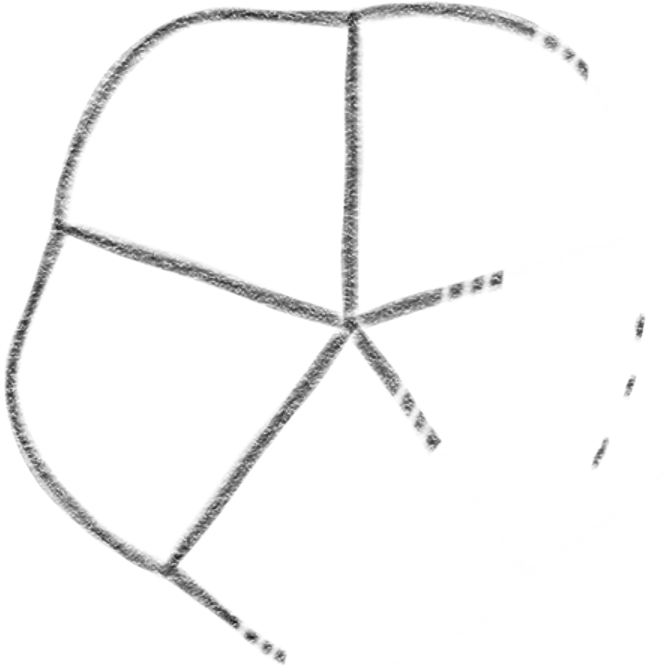}
\put(55,55){$v$}
\put(11,67){$v_1$}
\put(18,7){$v_2$}
\put(54,91){$v_g$}
\end{overpic}
\caption{Spine $\Gamma^g$.}
\label{fig:Gamma^g}
\end{subfigure}
\begin{subfigure}[b]{.32\linewidth}
\centering
\begin{overpic}[scale=.14,percent]{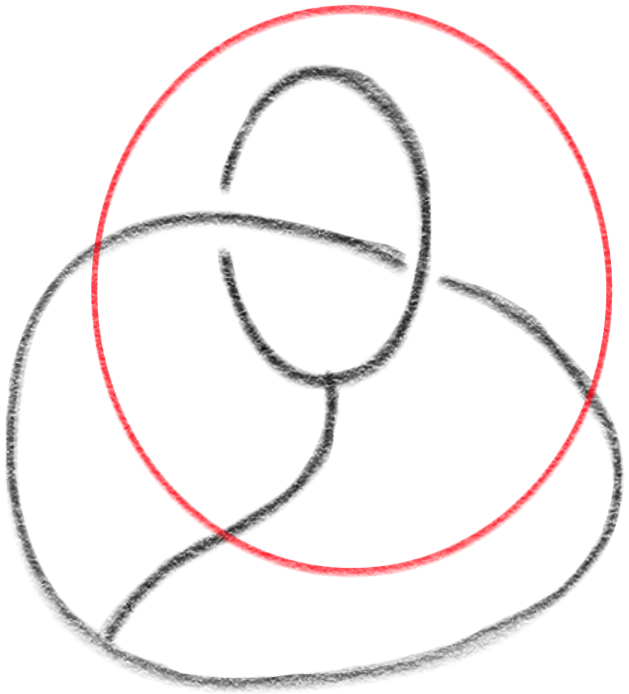}
\end{overpic}
\caption{$(\sphere,\Gamma_{X_i})$.}
\label{fig:GammaX}
\end{subfigure}
\caption{}
\end{figure}

Consider the genus $g>2$ handlebody-knot $\pairwheel$ and the set $\systemP$ of $3$-punctured spheres $P_1,\dots, P_g$ in Fig.\ \ref{fig:pairwheel}; see Fig.\ \ref{fig:wheelfive} for the example with $g=5$.
Denote by $X_i\subset\Compl \wheel$ the components cut off by $\uniP:=\cup_{i=1}^g P_i$ with $\partial_f X_i=P_i$, and by 
$Y\subset\Compl \wheel$ the 
component with $\partial_f Y=P_1\cup\dots\cup P_g$. 
There exist annuli $A_1,\dots, A_g$ in the exterior of $\wheel_Y$ as shown in Fig.\ \ref{fig:wheelY}.
The disks bounded by $\partial A_i$ induces a wheel-like spine $\Gamma^g$ of the handlebody $\wheel_Y\cup \rnbhd{\mathcal{A}}$, where $\mathcal{A}:=A_1\cup\dots\cup A_g$; see Fig.\ \ref{fig:Gamma^g}. 
Let $v$ be the $g$-valent node 
in the center of $\Gamma^g$, and $v_1,\dots, v_g$ the nodes around.
Since $\Gamma^g$ cannot be disconnected by cutting any edge, $A_i$ is incompressible in $\Compl{\wheel_Y}$.

\begin{lemma}\label{lm:example_p3_system}
$\systemP:=\{P_1,\dots, P_g\}$ is a maximal, full $\p 3$-system of $\pairwheel$.
\end{lemma}
\begin{proof}
It is clear that components of $\partial P_i$ are mutually parallel in $\partial\HK$, and 
hence $P_i$ is non-boundary-parallel.
Suppose $P_i$ admits a compressing disk $D$. Then $D$ cannot be in $X_i$ since 
the constituent link of $(\sphere,\Gamma_{X_i})$ 
is non-split; see Fig.\ \ref{fig:GammaX}. This implies $\partial D$ is disjoint from $\mathcal{A}$, and hence $D$ can be isotoped away from $\mathcal{A}$.
Thereby $D$ induces a $2$-sphere meeting $\Gamma^g$ at one point, an impossibility. 
Therefore $\systemP$ is a full $\p 3$-system. 

To see $\systemP$ is maximal, observe
first that if there is another $\p 3$-surface $Q$ disjoint from $\uniP$, then 
$Q$ cannot be in $X_i$ since the constituent link of 
$(\sphere,\Gamma_{X_i})$ is non-split.
Whence $\partial Q$ either is in the $g$-punctured sphere $S$ in $Y\cap \HK$ or meet an annulus in $Y\cap \HK$.
In the former, $Q$ can be isotoped away from $\mathcal{A}$, yet this implies that some $v_i,v_j\in\Gamma^g$, $i\neq j$, can only be connected via paths passing through $v$, a contradiction. 
In the latter, it may be assumed that $Q$ meets $A_1$ and hence $\partial Q,\partial P$ are parallel. Thus one can isotope $Q$ so that $Q\cap A_1$ is an arc, which cuts off from $A_1$ a disk $D$. The frontier of a regular neighborhood $D\cup Q\cup P_1$ induces a $2$-decomposing surface $F$ 
with parallel boundary components in $S$. Since there is no local knot in $\Gamma^g$, $F$ is boundary-parallel, and this implies $Q$ and $P_1$ are parallel. 
\end{proof}  
 
\subsection{Surfaces in $\p 3$-decomposable handlebody-knot exteriors}
Here we show that topological features of a $\pair$ can be recovered to a considerable extent from the $\p 3$-decomposition.
In the following, the boundary components of a $\p 3$-surface $P$ are labeled as in the beginning of Section \ref{subsec:arcs}.

\begin{lemma}\label{lm:disks}  
If $\Compl\HK$ is $\partial$-reducible, then $\Compl\HK$ 
admits an essential disk disjoint from
$\uniP$.
\end{lemma}
\begin{proof}
Choose an essential disk $D\subset\Compl\HK$ and 
choose $\uniP$ in its isotopy class so that they minimize $\vert D\cap \uniP\vert$.  
By the irreducibility of $\Compl\HK$ and the incompressibility of $\uniP$,
$D\cap \uniP$ contains no circles. 

Suppose $D\cap \uniP\neq\emptyset$.
Then there exists $P\in\systemP$ 
such that $D\cap P\neq\emptyset$.
Since components of $\partial P$ are mutually parallel, and each is separating in $\partial\HK$, we have $\vert\partial D\cap \partial P\vert=6k$ and hence $\vert D\cap P\vert=3k$, $k\in\mathbb{N}$.

Consider an outermost arc $\alpha$ in $\bd \alpha_D$, and let $D'\subset D$ be the disk it cuts off.
It may be assumed that $\alpha$ is an $\one$-arc or $\one^\star$-arc. If it is the former, then 
$D\cap P$ contains only $\alpha$ by Lemma \ref{lm:arcs}\ref{itm:no_adjacent_parallel}, contradicting $\vert D\cap P\vert\geq 3$.
If $\alpha$ is an $\one^\star$-arc, then 
$\alpha$ cuts off from $P$ a disk $D''$.
Pushing $D''\cup D'$ away from $P$ gives us a 
disk $D'''\subset\Compl\HK$. 
The disk $D'''$ cannot be essential since $\vert D'''\cap P\vert<\vert D\cap P\vert$, and hence $D'''$ cuts a $3$-ball $B$ off from $\Compl\HK$. 
If $\mathring{B}\cap\uniP=\emptyset$, 
then isotoping $P$ through $B$ reduces $\vert D\cap\uniP\vert$, a contradiction. If $\mathring{B}\cap\uniP\neq\emptyset$, then there is a arc $\beta\subset D\cap\uniP$ that cuts off a disk $D'_0$ from $D'$ with $\mathring{D'_0}\cap\uniP=\emptyset$. Suppose
$\beta$ is in $P'\subset\uniP$. Then $D'_0$ and the disk $D''_0\subset P'$ cuts off by $\beta$ bound a $3$-ball $B'$ with $\mathring{B'}\cap \uniP=\emptyset$. Isotopying $P'$ through $B'$ decreases $\vert D\cap\uniP\vert$, a contradiction.
Therefore, $D\cap P=\emptyset$.  
\end{proof}
 
\begin{corollary}\label{cor:genus_two_b_irre}
If $g=2$, then 
$\Compl\HK$ is $\partial$-irreducible.
\end{corollary} 
\begin{proof}
Since $g=2$, every $\uniP$ cuts 
$\partial\HK$ into some annuli and two once-punctured tori.
If $\Compl\HK$ is $\partial$-reducible, 
then by Lemma \ref{lm:disks}, there exists an essential disk $D$ disjoint from $\uniP$. 
In particular, $D$ is in an end $X$ and 
$\partial D$ is in the once-punctured torus $F\subset X\cap \partial\HK$. Denote by $V\subset\HK$ the solid torus cut off 
by a disk bounded by $\partial F$, and let $P\subset\uniP$ be the component meeting $X$. Then $D$ meets at one point the boundary $m$ of an essential disk of $V$ disjoint from $\partial F$, and the frontier of a regular neighborhood of $m\cup D$ in $X$ induces a compressing disk of $P$, a contradiction. 
\end{proof}


\begin{lemma}\label{lm:annuli} 
Suppose $\Compl\HK$ is $\partial$-irreducible, if $\Compl\HK$ admits an essential annulus (resp.\ a $2$-decomposing surface), then there exists an essential annulus (resp.\ $2$-decomposing surface) disjoint from $\uniP$. 
\end{lemma}
\begin{proof}
\cout{
Observe first that compressing a $2$-decomposing surface along a $\partial$-compressing disk induces an essential disk in $\Compl\HK$, so the $\partial$-irreducibility of $\Compl\HK$ implies that every $2$-decomposing surface is $\partial$-incompressible, and hence essential. 
Choose a $2$-decomposing surface $A$ and choose $\uniP$ in its isotopy class so that the total index is minimized.  
}
Note first that by the $\partial$-irreducibility, every incompressible, non-boundary parallel annulus in $\Compl\HK$ is $\partial$-incompressible and hence essential.
Choose an essential annulus (resp.\ $2$-decomposing surface) $A$ and choose $\uniP$ in its isotopy class so that their total index is minimized.   
%
Then, by Lemma \ref{lm:mini_index}, for any $P\in \systemP$, $A,P$ minimizes the partial index $\pind$ in their isotopy classes. In particular, $n_\otimes=0$ by Lemma \ref{lm:two_incompressible}. 
If there exists an inessential arc $\alpha_0$ in $\alphaA$.
Then by Lemma \ref{lm:arcs}\ref{itm:outermost}\ref{itm:no_adjacent_parallel}, $\alpha_0$ is outermost.    
Denote by $C_1',C_2'$ the components of $\partial A$. Then it may be assumed that $\alpha_0$ is a $\one$-arc with $\partial \alpha_0\subset C_1'$. Therefore there are
three essential arcs $\alpha_1,\alpha_2,\alpha_3$ in $\alphaA$ 
such that $\alpha_i\cap C_1'$ is adjacent to 
$\alpha_{i+1}\cap C_1'$, $i=0,1,2$. In particular, $\alpha_1\cap C_1'$ is labeled with $2$. Note that $\alpha_1\cap C_2'$ cannot be labeled with $2$ or $3$ since $A$ is $\partial$-incompressible and $\alpha_0$ is a $\one$-arc. If $\alpha_1\cap C_2'$ is labeled with $\one$, then either $\alpha_2$ or $\alpha_3$ is a $(\two,\three)$-arc, again contracting $\alpha_0$ is a $\one$-arc. Therefore, $\alphaA$ contains only essential arcs.

Let $\alpha_1,\alpha_2$ (resp.\  $\alpha_2,\alpha_3$) are immediately parallel essential arcs in $\alphaA$. 
Since the labeling on $C_1',C_2'$ follow the pattern in \eqref{eq:pattern}, it may be assumed that $\alpha_i\cap C_1'$ is labeled with 
$i$, $i=1,2,3$. If $\alpha_1$ is a $\one$-arc, then either 
$\alpha_2$ is a $\two^\ast$-arc or $\alpha_3$ is a $(\two,\three)$-arc. The former contradicts $A$ being $\partial$-incompressible, while the latter contradicts $\alpha_1$ being a $\one$-arc.
Similarly, if $\alpha_1$ is a $(\one,\two)$-arc, 
then either $\alpha_3$ is a $\three^\ast$-arc or 
$\alpha_2$ is a $(\one,\two)$-arc; the former cannot happen by the $\partial$-incompressibility of $A$, while the letter contradicts Lemma \ref{lm:arcs}\ref{itm:not_parallel_types}. 
Lastly, if $\alpha_1$ is a $(\one,\three)$-arc, 
then either $\alpha_2,\alpha_3$ both are $(\two,\three)$-arcs or $\alpha_2$ is a $\two^\ast$-arc, yet neither happens as in the previous case. As a result, $A\cap P$ contains only essential loops.

If $A\cap P\neq \emptyset$, then there exists an outermost loop in $\alphaA$, which cuts off an outermost annulus $A'$ from $A$ and an annulus $A''$ from $P$. 
Then the union $A'\cup A''$ induces 
an incompressible annulus (resp.\ $2$-decomposing sphere) 
with smaller total index with $\uniP$ than $A$ does, a contradiction.
Thus $A\cap P=\emptyset$, for every $P\in\systemP$, and hence $A\cap\uniP=\emptyset$.  
\end{proof}

\begin{lemma}\label{lm:tori} 
If $\Compl\HK$ contains an incompressible torus, then there exists an incompressible torus in $\Compl\HK$ disjoint from
$\uniP$.
\end{lemma}
\begin{proof}
Let $T$ be an incompressible torus that 
minimizes $\vert T\cap \uniP\vert$ among 
all incompressible tori in $\Compl\HK$. 
If $T\cap\uniP\neq 0$, then 
there exists $P\in \systemP$ such that 
$T\cap P\neq\emptyset$. 
Label the components of $\partial P$ and the two components cut off by $P$ from $\Compl\HK$ as in Fig.\ \ref{fig:P_cuts_surface}. 
By the incompressibility of $P$ and $T$, $P$ cuts $T$ into some annuli in $X,Y$, and every boundary component of the annuli is parallel to $C_1,C_2$ or $C_3$. 

Let $A\subset T$ be an annulus $A\subset T$ cut off by $P$. Consider first the case $\partial A$ is parallel in $P$. This implies $\partial A$ cut an annulus $A'$ off from $P$. 
The union $A'\cup A$
induces a torus $T'$, and by the minimality, $T'$ bounds a solid torus $V$ in $X$ or $Y$. If $T\cap \mathring{V}=\emptyset$, 
then $T$ can be isotoped so $\partial A$ 
is removed from the intersection, contradicting the minimality. 
If $T\cap \mathring{V}\neq\emptyset$, 
then let $A''\subset T\cap V$ an outermost annulus in $V$ that cuts off a solid torus $W$ disjoint from $A$.
Compressing $A''$ through $W$ yields an isotope that decreases $T\cap \uniP$, 
a contradiction. Therefore, every annulus $A\subset T$ cut off by $P$ has boundary components non-parallel in $P$.  
This implies if $A\subset X$ (resp.\ $Y$), then the components of $\partial A$ are parallel to $C_2,C_3$ (resp.\ $C_1,C_2$), respectively. However, one of the annuli in $T$ cut off by $P$ that are adjacent to $A$ has both boundary components parallel to $C_3$ (resp.\ $C_1$) in $P$, a contradiction. 
\end{proof}
 
\subsubsection*{Example}
Recall the $\p 3$-decomposable genus $g>2$ handlebody-knot $\pairwheel$ and the maximal $\p 3$-system $\systemP$ 
of $\pairwheel$ in Fig.\ \ref{fig:pairwheel}. 
We use the same notation as in Lemma \ref{lm:example_p3_system}. 

\begin{lemma}\label{lm:example_birre_atoro_twoindecom}
The exterior $\Compl{\wheel}$ is $\partial$-irreducible and atoroidal, and $\pairwheel$ is $2$-indecomposable.
\end{lemma}
\begin{proof}
Suppose $\Compl\wheel$ is not $\partial$-irreducible. Then by Lemma \ref{lm:disks}, there exists an essential $D\subset\Compl\wheel$ disjoint from $\uniP$. The disk $D$ cannot be in $X_i$ since the constituent link of $(\sphere,\Gamma_{X_i})$ is non-split; see Fig.\ \ref{fig:GammaX}. 
Therefore, $\partial D$ is in the $g$-punctured sphere of $Y\cap\wheel$. 
In particular, $D$ can be isotoped away from $\mathcal{A}:=A_1\cup\dots\cup A_g$, yet this implies $\Gamma^g-v$ is disconnected, a contradiction.

Suppose $\Compl\wheel$ is toroidal. 
Then by Lemma \ref{lm:tori}, there exists an incompressible torus $T$ disjoint from $\uniP$. The torus $T$ cannot be in $X_i$ since $X_i$ is a handlebody, so $T\subset Y$. One can isotope $T$ so that $\vert T\cap \mathcal{A}\vert$ is minimized. If $T\cap \mathcal{A}\neq\emptyset$, 
then $\mathcal{A}$ cuts $T$ into some 
annuli. Let $A'$ be one of the annuli, and observe that $\partial A'$ meets different components of $\mathcal{A}$ by the minimality. In particular, components of $\partial A'$ bound disks dual to different edges in $\Gamma^g$. This implies that cutting two edges of $\Gamma^g$ disconnects $\Gamma^g$, a contradiction. As a result $T$ is disjoint from $\mathcal{A}$, and hence 
in the exterior of $\Gamma^g$, which is not possible, given $\Compl{\Gamma^g}$ is a handlebody.

Suppose $\pairwheel$ is $2$-decomposable.
Then by Lemma \ref{lm:annuli}, there exists a $2$-decomposing surface $F$ disjoint from $\uniP$. By the atoroidality of $\pairwheel$, 
components of $\partial F$ are not parallel in $\partial \wheel$, so $\partial F$ 
is in the $g$-punctured sphere
$S$ in $Y\cap \wheel$. If the two
components of $\partial F$ are parallel to two components of $\partial S$, 
then cutting two edges adjacent to $v$ disconnects $\Gamma_g$, which is not possible.
If only one component of $\partial F$ parallel to a component of $\partial S$, 
then removing $v$ and one adjacent edge disconnects $\Gamma_g$, again an impossibility. 
Lastly if no component of $\partial F$ 
parallel to a component of $\partial S$, 
then removing $v$ disconnects $\Gamma_g$, 
yet this happens neither.    
\end{proof}

\subsubsection{Hyperbolic handlebody-knots with homeomorphic exteriors}\label{subsubsec:hyp}  
The handlebody-knot $\hpairn$ in Fig.\ \ref{fig:hyp_n} is obtained by twisting the twice-punctured disk $S$ in the exterior of $(\sphere,H)$ in Fig.\ \ref{fig:hyp_zero} $n$ times (see Fig.\ \ref{fig:signs} for the convention), $n\in\mathbb{Z}$, and hence its exterior is homeomorphic to $\Compl H$, for every $n$. 

It is not difficult to see the $3$-punctured sphere $P$ in Fig.\ \ref{fig:hyp_n} is a $\p 3$-surface in $\Compl {H_n}$, and no incompressible torus or essential annuli in $\Compl {H_n}$ disjoint from $P$. Thus $\hpairn$, $n\in\mathbb{Z}$, are all hyperbolic by Corollary \ref{cor:genus_two_b_irre} and Lemmas \ref{lm:annuli} and \ref{lm:tori}.

To see members in $\{\hpairn\}_n$ are mutually inequivalent, we observe first $\{P\}$ is a maximal $\p 3$-system. Denote by $L_n$ the constituent links of the induced spatial graph $(\sphere, \Gamma_n)$ given by $(\sphere,H_n^X)$ and $P$ (see Fig.\ \ref{fig:induced_sg_link}), and note that the Alexander polynomial of $L_n$ is $(n+1)t-2n+(n+1)t^{-1}$. Comparing the Alexander polynomials, we see $\hpairn,\hpairm$ and hence $L_n,L_m$ are equivalent if and only if $n=m$. 

\begin{figure}[t]
	\begin{subfigure}[b]{.31\linewidth}
		\centering
		\begin{overpic}[scale=.15,percent]{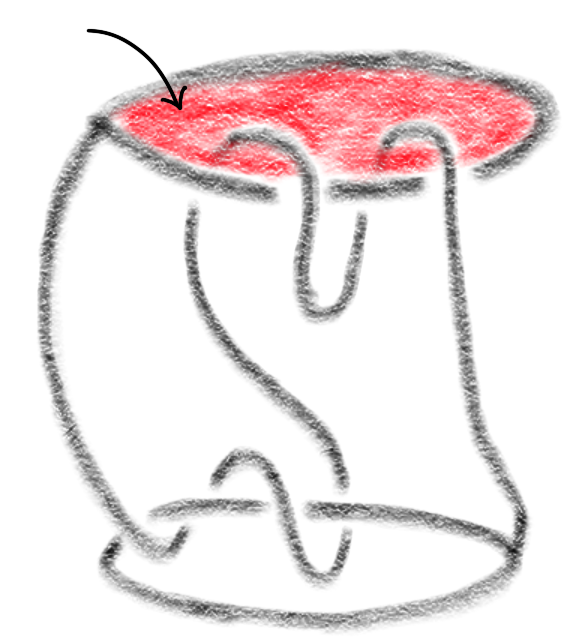}
		\put(2,92){\large $S$}
		\end{overpic}
		\caption{$(\sphere,H)$}
		\label{fig:hyp_zero}
	\end{subfigure}
	\begin{subfigure}[b]{.32\linewidth}
		\centering
		\begin{overpic}[scale=.15,percent]{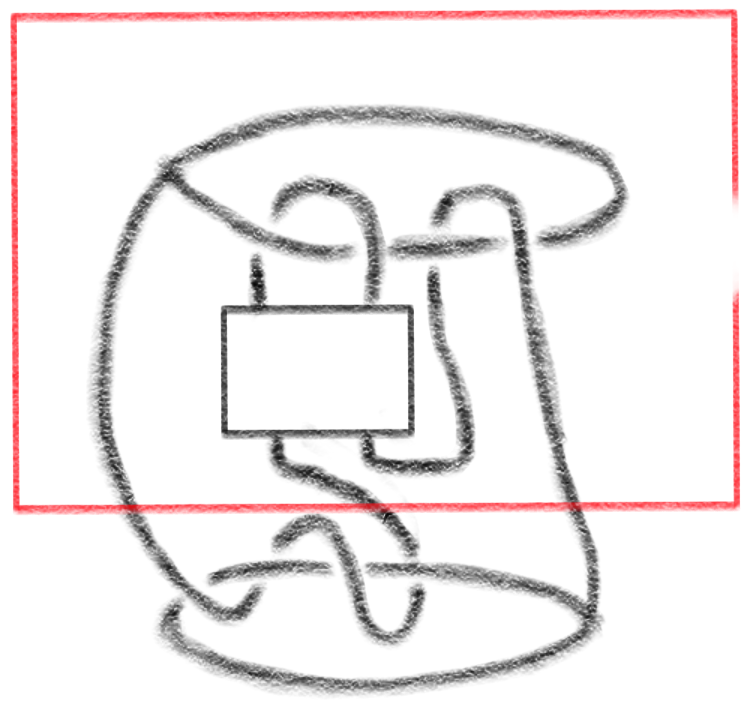}
			\put(95,57){\large $P$}
			\put(40,43){\Large $n$}
			\put(88,10){\large $X$}
		\end{overpic}
		\caption{$\hpairn$}
		\label{fig:hyp_n}
	\end{subfigure}
	\begin{subfigure}[b]{.32\linewidth}
		\centering
		\begin{overpic}[scale=.13,percent]{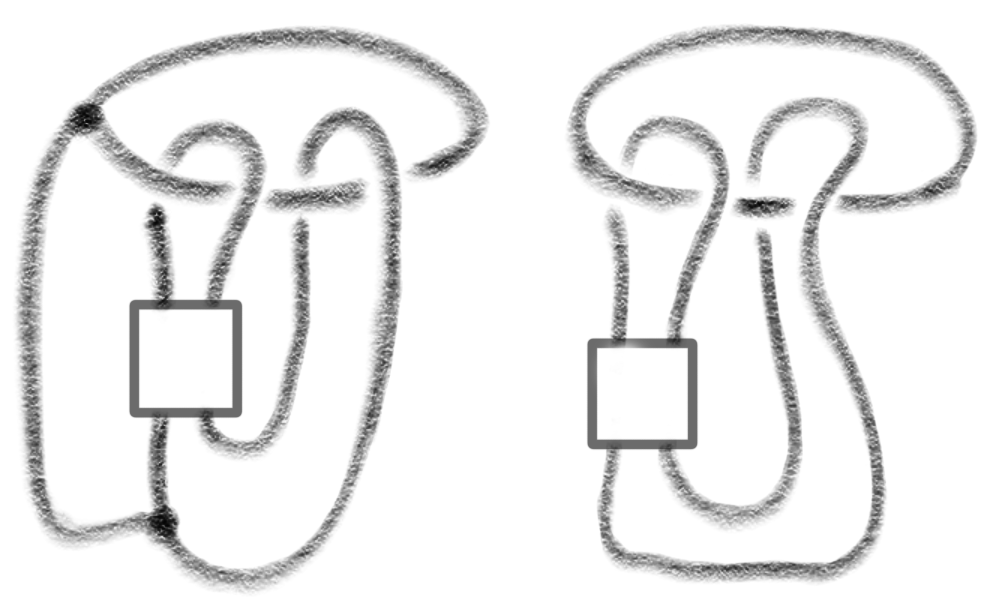}
			\put(16,22){\Large$n$}
			\put(62,18){\Large$n$}
		\end{overpic}
		\caption{$(\sphere,\Gamma_n)$ and $L_n$.}
		\label{fig:induced_sg_link}
	\end{subfigure}
		\begin{subfigure}[b]{.26\linewidth}
		\centering
		\begin{overpic}[scale=.15,percent]{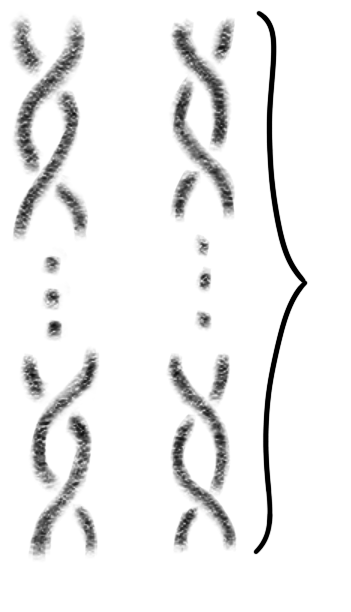}
			\put(51.3,52){$n$}
			\put(-2,2){\footnotesize $n>0$}
			\put(22,2){\footnotesize $n<0$}
		\end{overpic}
		\caption{Signs.}
		\label{fig:signs}
	\end{subfigure}
	\begin{subfigure}[b]{.35\linewidth}
	\centering
	\begin{overpic}[scale=.16,percent]{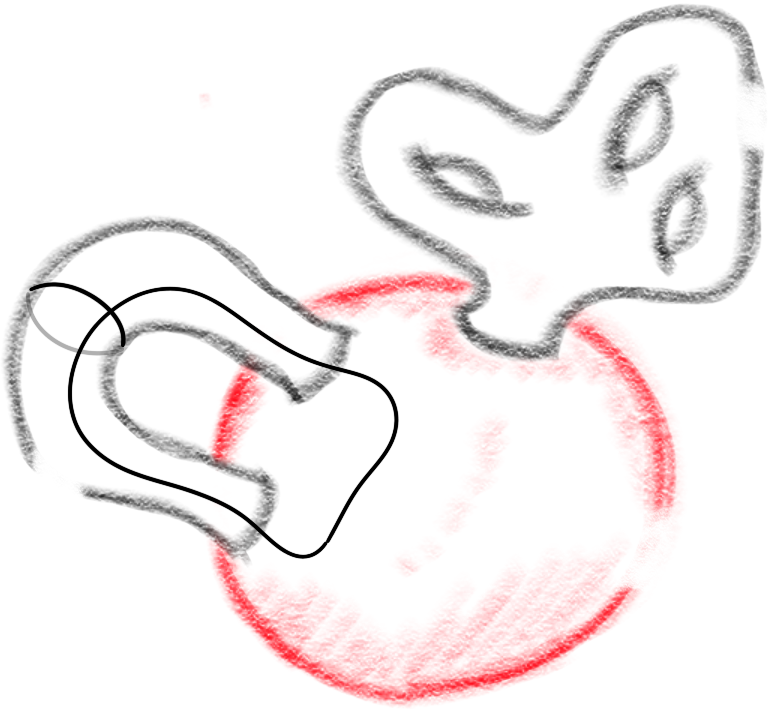}
		\put(20,75){$X$}
		\put(53,17){$\Compl X$}
		\put(3.5,49){\small $m$}
		\put(22,55){\small $l$}
		\put(13.9,54.3){\footnotesize $p$}
		\put(3,27){$A$}
		\put(94,73){$F$}
		\put(81,17){$P$}
	\end{overpic}
	\caption{Geodesics $m,l\subset\partial \HK_X$.}
	\label{fig:geodesics_X}
	\end{subfigure}
	\begin{subfigure}[b]{.35\linewidth}
	\centering
	\begin{overpic}[scale=.16,percent]{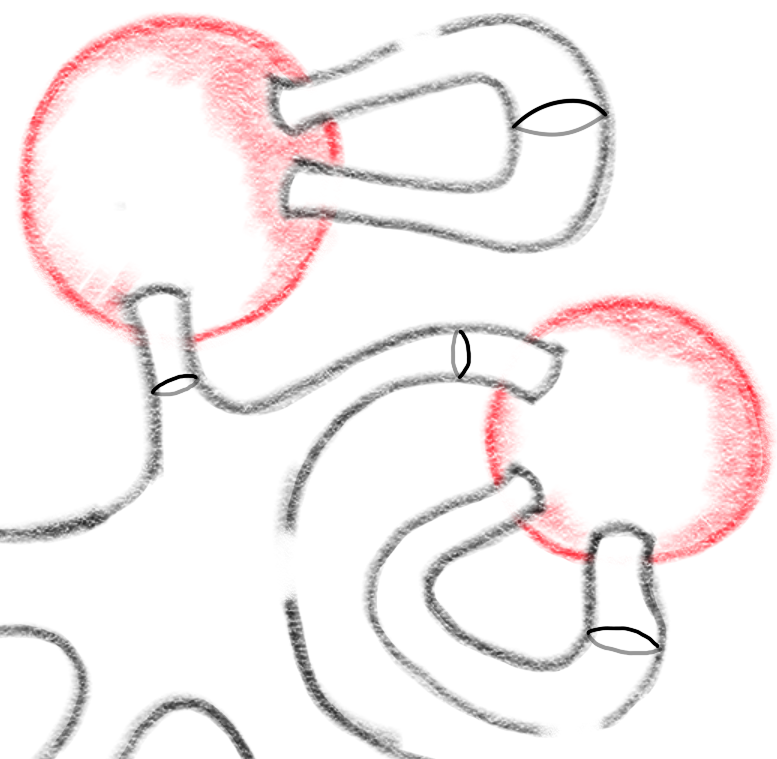}
		\put(10,70){$X_1$}
		\put(80,40){$X_2$}
		\put(18,18){$Y$}
		\put(80,82){$\alpha_1$}
		\put(86.5,12){$\alpha_2$}
		\put(12,45){$\beta_1$}
		\put(55,42){$\beta_2$}
		\put(35,22){$S$}
		\put(50.7,90.5){$A_1$}
		\put(70,0.5){$A_2$}
	\end{overpic}
	\caption{Geodesics $\alpha_i,\beta_i\subset\partial\HK_Y$.}
	\label{fig:geodesics_Y}
	\end{subfigure}
	\caption{}
\end{figure}

\section{Symmetry of $\p 3$-decomposable handlebody-knots}\label{sec:applications}

\subsection{Handlebody-knot symmetry}\label{subsec:symmetry}
Throughout the subsection, we assume that $\pair$ is $2$-indecomposable, $\p 3$-decomposable with $\systemP$ a maximal, semi-full $\p 3$-system, and $\Compl\HK$ is $\partial$-irreducible. Recall that 
$\uniP$ denotes the union of members in $\systemP$, and $e$ is the number of ends in the $\p 3$-decomposition. 

Let $M$ be a $3$-manifold and $X_1,\dots, X_n, Y$ are subpolyhedra in $M$. Then 
\[\Aut{M,X_1,\dots,X_n, \rel Y}\quad
\big(\text{resp. } \pAut{M,X_1,\dots,X_n,  \rel Y}\big)\]
denotes the space of (resp.\ orientation-preserving) self-homeomorphisms of $M$ preserving $X_1,\dots, X_n$ and fixing $Y$. The corresponding mapping class group is denoted by 
\[\MCG{M,X_1,\dots,X_n, \rel Y}\quad \big(\text{resp. } \pMCG{M,X_1,\dots,X_n, \rel Y}\big).\]
\subsection{Known results}
The uniqueness of $\systemP$ given by Theorem \ref{teo:uniqueness}, together with \cite{Hat:76} (see also \cite{Hat:99}) implies the following.

\begin{lemma}\label{lm:fix_surface}
The natural homomorphisms 
\[\Sym{\HK,\uniP}\rightarrow \Sym{\HK}, \quad \pSym{\HK,\uniP}\rightarrow \pSym{\HK}\]
are isomorphisms.
\end{lemma}  
 
Recall that a \emph{cut system} of a handlebody is a set of disjoint essential disks in the handlebody that cuts the handlebody into some $3$-balls. A cut system induces 
a spine of the handlebody and vice versa.  
If $\mathcal{D}$ is a cut system of 
$\HK$ and $\Gamma_\mathcal{D}$ the induced spine, then we have the isomorphisms (see \cite[Lemme $2.2$]{ChoKod:13})
\[\Sym{\HK,\uniD}\simeq \Sym{\Gamma_\uniD}, \quad \pSym{\HK,\uniD}\simeq \pSym{\Gamma_\uniD}.\]

Denote by $\uniD_\uniP$ the union of disjoint disks in $\HK$ with $\partial\uniD_\uniP= \partial \uniP$, and observe that if $e=g$, namely $\systemP$ is full, then for each end $X$, $X\cap\HK$ consists of a once-punctured torus $F$ and an annulus. There is therefore, up to isotopy, a unique disk $D_X$ in $\HK$ bounded by a non-separating loop in $F$ with $D_X\cap \uniD_\uniP=\emptyset$.  
The union $\tilde\uniD_\uniP:=\cup_X D_X\cup\uniD_\uniP$ is a cut system of $\HK$, and determines a spine $\sghk$ of $\HK$. 
\begin{lemma}\label{lm:finiteness}
If $\pair$ is fully $\p 3$-decomposable and atoroidal, then $\Sym\HK$ is finite.
\end{lemma}
\begin{proof}
By Lemma \ref{lm:fix_surface}, we have 
$\Sym\HK\simeq \Sym{\HK,\uniP}$ and since $\systemP$ is unique, $\Sym{\HK,\uniP}\simeq \Sym{\HK,\tilde \uniD_\uniP}$.
Therefore $\Sym\HK\simeq \Sym{\sghk}$. The assertion then follows from \cite[Theorems $2.5$, $3.2$]{ChoKod:13}.
\end{proof}
\cout{
On the other hand, if $g=2$, the result of Funayoshi-Koda \cite{FunKod:20} implies the following.
\begin{lemma}\label{lm:finiteness_genustwo}
Given a genus two non-trivial $\pair$, 
$\pair$ is atoroidal if and only if $\Sym\HK$ is finite. 
\end{lemma}
}


\begin{lemma}\label{lm:surface_involutions}
Let $F$ be a closed surface of genus $g>1$, and $f,g$ two involutions of $F$.
If $f$ and $g$ are isotopic, then 
there exists an isotopy $\phi_t\in \Aut{F}$ such that $f=\phi_1 g\phi_1^{-1}$ and $\phi_0=\id$.
\end{lemma}
\begin{proof}
It is a special case \cite[Theorem $6.1$]{Tol:81}.
\end{proof}

As a corollary of \cite[Addendum]{Zim:82} and \cite[Theorem $7.1$]{Wal:68}, we have the following.
\begin{lemma}\label{lm:involution}
Given $h\in \Aut{\Compl\HK}$, if $h^2$ is isotopic to $\id$, then there exists an involution $g$ of $\Compl\HK$ isotopic to $h$.
\end{lemma}

\subsection{Symmetries of finite order} 
Throughout the rest of the section, 
$H$ is a finite subgroup of $\Sym\HK$, and we set $H_+:=H\cap\pSym\HK$, and if $H_+\neq H$, define $H_-:=H-H_+$.
Consider the homomorphisms
\[
\pMCG{\sphere,\HK}\xrightarrow{\phi}\pMCG{\HK},\quad \pMCG{\HK}\xrightarrow{\psi} \pMCG{\partial\HK},
\]
given by the restriction maps.
\begin{lemma}\label{lm:pair_surface_injective}
The homomorphism $\psi\circ\phi$ restricts to an injection on $H$.   
\end{lemma}
\begin{proof}
It is known that $\Aut{\HK,\rel\partial \HK}$ is contractible by \cite{Hat:76} and Alexander trick, so $\psi$ is injective.  
Since $\pi_1(\Aut{\partial\HK})$ is trivial by \cite[Theorem $3.1$]{Sco:70}, we have $\pi_1(\Aut{\HK})$ is also trivial. 
Therefore, there is an exact sequence
\[0\rightarrow\MCG{\Compl\HK,\rel\partial \Compl\HK}\rightarrow\MCG{\sphere,\HK}\xrightarrow{\phi}\MCG{\HK}.\]
Since $\MCG{\Compl\HK,\rel\partial \Compl\HK}$ is torsion free by \cite{HatMcC:97}, $\phi\vert_H$ is injective.
\end{proof}

\begin{lemma}\label{lm:orien_pre_some_end}
Given $f\in H_+$, if 
$f(X)=X$, for some end $X$, 
then $f\simeq \id$.
\end{lemma}
\begin{proof}
By Lemma \ref{lm:fix_surface}, it may be assumed that $f(P)=P$. Since $f$ does not permute components of $\partial P$, one may further assume $f\vert_P=\id$. 
In particular, $f$ induces homeomorphisms
$f_X\in\pAut{\sphere,\hkX,\rel P},
f^X\in\pAut{\sphere,\hkx, \rel P}$. 

The mapping class $[f]$, being of finite order, implies $[f_X]$ and hence $[f_X\vert_{\partial\hkX}]$ are of finite order. Applying the 
Nielsen realization \cite{Nil:80}, \cite{Nil:83}, one finds a hyperbolic structure on $\partial\hkX$ and an isometry $g$ isotopic to 
$f_X\vert_{\partial\hkX}$ in $\pAut{\partial\hkX}$.  

Denote by $A,F$ the annulus and once-punctured surface in $X\cap\HK$ and $A',F'$ the ones in 
$\Compl X\cap\HK$.
Let $m$ be the geodesic isotopic to $\partial A$ and $l$ the geodesic isotopic to a loop meeting $m$ at one point and disjoint from the geodesic isotopic to $\partial F$ (see Fig.\ \ref{fig:geodesics_X}). Then the loops $m,l$ are preserved by $g$. Since 
$m\cap l$ is a point $p$, the two loops $m,l$ are actually fixed by $g$, so $g$ fixes a frame at $p$. Therefore $g=\id$. 

Consider now the fibration sequence given by the restriction maps:
\[
\pAut{\partial\hkX,\rel P}
\rightarrow
\pAut{\partial\hkX}
\rightarrow
\Emb{P,\partial\hkX},
\]
where $\Emb{P,\hkX}$ is the component of the space of embeddings of $P$ in $\partial\hkX$ that 
contains the inclusion.
By \cite[Theorem $1.1$]{Yag:05}, 
the space $\Emb{P,\partial\hkX}$ is contractible, so
\[
\pMCG{\partial\hkX,\rel P}\rightarrow \pMCG{\partial\hkX}
\] 
is an isomorphism. 
Since $f_X\vert_{\partial\hkX}$ and $g=\id$ are isotopic in $\pAut{\partial\hkX}$, 
we conclude that $f_X$ and $\id$ are isotopic in $\pAut{\partial\hkX, \rel P}$. 

With $A,F$ replaced by $A',F'$, the above argument applies to $\hkx$ and shows that 
$f^X$ is isotopic to $\id$ in 
$\pAut{\partial\hkx, \rel P}$, 
and hence $f\vert_{\partial\HK}$ is isotopic to $\id$ 
in $\pAut{\partial\HK}$. 
The assertion thence follows from Lemma \ref{lm:pair_surface_injective}.
\end{proof}

\begin{lemma}\label{lm:orien_rev_some_end}
If $f\in H_-$ and $f(X)=X$, for some end $X$, then $\vert\systemP\vert=1$.
\end{lemma}
\begin{proof}
Let $X:=X_1,\dots, X_e$ be the ends of the $\p 3$-decomposition. Then it suffices to show $e=2$ and $\partial_f X_1=\partial_f X_2$. Suppose otherwise. Then 
$Y:=\Compl\HK-\cup_{i=1}^e \mathring{X_i}$ is non-empty. By Theorem \ref{teo:ends_connectors}, there are two possibilities: 
\begin{enumerate}[label=\textnormal{(\roman*)}]
\item\label{itm:cuff_like} $Y\cap\HK$ consists of $e$ annuli $A_1,\dots, A_e$ and an $e$-punctured sphere $S$ with $\partial A_i\subset \partial_f X_i$;
\item\label{itm:theta} $e=2$ and $Y\cap \HK$ consists of three annuli $A_1,A_2,A_3$ each of which meet both $X_1,X_2$. 
\end{enumerate}   
Note that the mapping class $[f]\in\Sym\HK$ is of order two by Lemma \ref{lm:orien_pre_some_end}, so the induced mapping class $[f_Y]\in\Sym\hkY$ is of order two. By the Nielsen realization, there exists a hyperbolic structure on $\partial\hkY$ such that $f_Y\vert_{\partial\hkY}$ is isotopic to an isometry $g_0$ with $g_0^2=\id$.

Consider first the case \ref{itm:cuff_like}, and denote by $C_i$, $i=1,\dots,e$, the components of $\partial S$. Let $\alpha_i,\beta_i\subset \partial\hkY$ be the geodesics isotopic to $\partial A_i, C_i$, respectively (see Fig.\ \ref{fig:geodesics_Y}). 
For each $\alpha_i,\beta_i, i=1,\dots,e$, we choose a disk in $\HK$ bounded by it,
and denote by $\uniD$ the union of these disks. Let $\sigma_i$ (resp.\ $\hat\sigma_i$) be the loop (resp.\ edge) dual to the disk bounded by $\alpha_i$ (resp.\ $\beta_i$). Then $\cup_{i=1}^e(\sigma_i\cup\hat\sigma_i)$
is the spine $\sgY$ of $\hkY$. Set $v:=\cap_i \hat\sigma_i$, $v_i:=\sigma_i\cap\uniD$, and $\hat v_i:=\hat\sigma_i\cap\uniD$. 

Since $g_0$ is an involution and $f$ preserves $X_1\cup\dots\cup X_e$, $g_0$ either preserves $\alpha_i$ 
or swaps $\alpha_i,\alpha_j$, $i\neq j$.
Thus $g_0$ can be extended to an involution $g_1$ of the union $\partial\HK\cup\uniD$ so that $g_1$ either fixes $v_i\cup \hat v_i$ or swaps it for $v_j\cup\hat v_j$, $j\neq i$. 
Since $\mathcal{D}$ cuts $\hkY$ into some $3$-balls, $g_1$ can be further extended, by the Alexander trick, to an involution $g_2$ on $\hkY$ so that $g_2$ either fixes $\sigma_i\cup\hat \sigma_i$ or swaps it 
for $\sigma_j\cup\hat\sigma_j$, for some $j\neq i$ (see Fig.\ \ref{fig:fix_swap}). 

By Lemma \ref{lm:involution}, there is an involution $g_3$ of $\Compl \hkY$ isotopic to $f$. The restriction $g_3\vert_{\partial \hkY}$ might not be identical to $g_0$, but they are isotopic. Hence by Lemma \ref{lm:surface_involutions}, there 
exists an ambient isotopy $\mu_t$ 
such that $\mu_t g_0 \mu_t^{-1}$ with $\mu_0=\id$ and $\mu_1 g_0 \mu_1^{-1}=g_3\vert_{\partial \hkY}$.  Let $\rnbhd{\partial \hkY}\simeq (\partial \hkY)\times [0,1]$ be a collar neighborhood
of $\partial \hkY\subset\Compl\HK$. Define $g_4$ to the composition  
\[
\overline{\Compl \hkY-\rnbhd{\partial \hkY}}\simeq \Compl\hkY\xrightarrow{g_3}\Compl\hkY \simeq \overline{\Compl \hkY-\rnbhd{\partial \hkY}},\]
together with $g_4(x,t):=\mu_t g_0 \mu_t^{-1}$, for $(x,t)\in (\partial \hkY)\times [0,1]\simeq \rnbhd{\partial\hkY}$. Then $g_4$ is an involution of $\Compl\hkY$ with $g_4\vert_{\partial\hkY}=g_0$.

Gluing $g_4, g_2$ along $\partial \hkY$  gives us an involution $g$ of $\sphere$ that preserves $\hkY$ and $\sgY$. 
Every orientation-reversing involution of $\sphere$ is conjugate to a reflection, thus the fixed point set of $g$ is either a $2$-sphere or a point. 
Since $g$ fixes $\sigma_1\cup\hat\sigma_1$, so it is the former; let $S$ be the fixed point set. Given $g$ either fixing $\sigma_i\cup\hat\sigma_i$ or swapping $\sigma_i\cup\hat\sigma_i,\sigma_j\cup\hat\sigma_j$, $i\neq j$, we have either $\sigma_i\cup\hat\sigma_i\subset S$ or  $S\cap (\sigma_i\cup\hat\sigma_i)=v$.
This implies, if $D$ is the disk bounded by $\sigma_1$ not containing $v$, then $\mathring{D}$ is disjoint from $\sgY$.   
Let $D':=Y\cap D$. Then the frontier of a regular neighborhood of the union of $D'$ and a component of $\partial A_1$ induces a compressing disk of $\partial_f X_1$, contradicting $\partial_f X_1$ is a $\p 3$-surface.

The above construction applies to the case \ref{itm:theta}, and thus $g_0$ can be extended to an involution on $\inducedhkY$ that fixes $\sgY$. This implies $\inducedsgY$ is planar and hence trivial, contradicting that 
$\partial_f X_1,\partial_f X_2$ are non-parallel.
\end{proof}
\begin{figure}[h!]
\begin{subfigure}{.48\linewidth}
\centering
\begin{overpic}[scale=.17,percent]{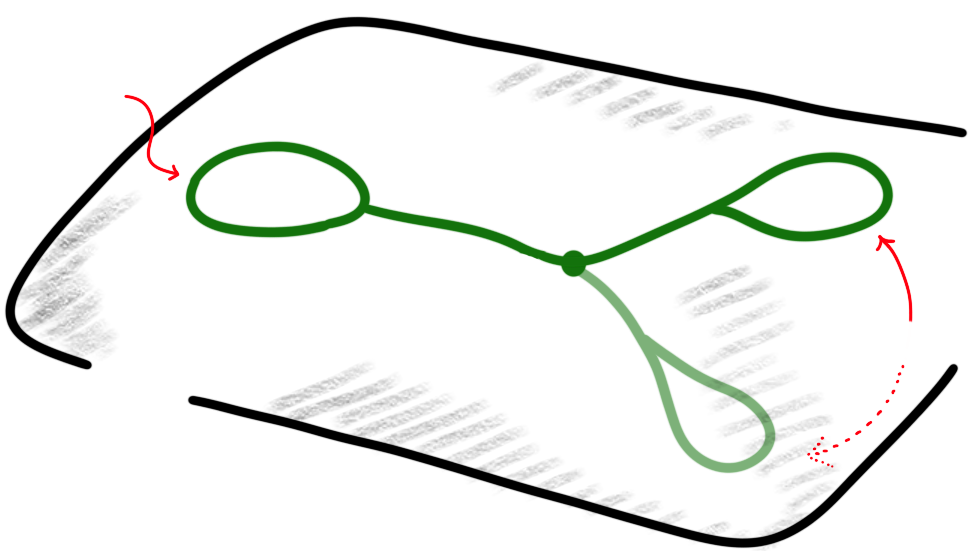}
\put(12,15){$S$}
\put(0.5,45.5){\footnotesize{fixed}}
\put(80.5,20){\footnotesize{swapped}}
\put(24,35){\footnotesize $D$}
\put(29,42.5){\footnotesize $\sigma_1$}
\put(45,34.6){\footnotesize $\hat\sigma_1$}
\put(90,40){\footnotesize $\sigma_i$}
\put(64,35){\footnotesize $\hat\sigma_i$}
\put(74,4.7){\footnotesize $\sigma_j$}
\put(59,20.6){\footnotesize $\hat\sigma_j$}
\put(57.4,31.4){\footnotesize $v$}
\end{overpic}
\caption{Fix $\sigma_1$; swap $\sigma_i,\sigma_j$.}
\label{fig:fix_swap}
\end{subfigure}
\begin{subfigure}{.48\linewidth}
\centering
\begin{overpic}[scale=.17,percent]{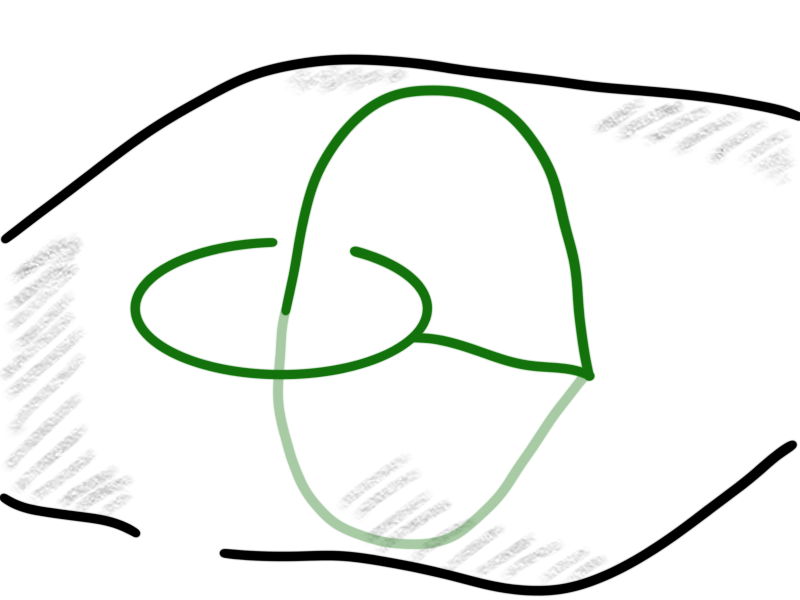}
\put(20,5){$S$}
\put(75,27.5){\small $v_1$}
\put(27.5,35.2){\small $v_2$}
\put(9,39){\small $\sigma_a$}
\put(59,32.7){\small $\sigma_f$}
\put(70.3,52){\small $\sigma_m$}
\end{overpic}
\caption{$\twoone$ and reflection sphere $S$.}
\label{fig:twoone_reflection}
\end{subfigure}
\caption{} 
\end{figure}

\subsection{Genus two classification} 
Following \cite{Mor:09}, we denote by $\twoone$, the spatial graph in Fig.\ \ref{fig:GammaX}.
\label{subsubsec:genus_two_classification}
\begin{lemma}\label{lm:g_two_orien_rev_one_end}
Suppose $g=2$, and $\Compl\HK$ is atoroidal.
If $f(X)=X$, for some end $X$ and $f\in \nAut{\sphere,\HK}$, then $\inducedsgX\simeq \twoone$.
\end{lemma}
\begin{proof}
Recall that by \cite{FunKod:20}, $\Sym\HK$ is finite.
By Lemma \ref{lm:orien_pre_some_end}, 
$[f]$ is of order two, and hence 
$[f_X]\in \MCG{\sphere,\hkX}$ is of order two. By the Nielsen realization, there exists a hyperbolic structure on $\partial\hkX$ such that $f_X\vert_{\partial\hkX}$ is isotopic to an isometry $g_0$ with $g_0^2=\id$. 

Since $X\cap\HK$ is an annulus $A$ and a once-punctured torus $F$, there is, up to isotopy, a unique non-separating loop $m$ in $F$ disjoint from $\partial F$ and bounds a disk in $\hkX$. Let $\sigma_a$ (resp.\ $\sigma_f$) be the loop (resp.\ the edge) dual to a disk bounded by the geodesic parallel to $\partial A$ (resp.\ $\partial F$) and by $\sigma_m$ the loop dual to a disk in $\hkX$ bounded by the geodesic parallel to $m$. The union $\sigma_a\cup\sigma_f\cup\sigma_m$ gives us the canonical spine $\sgX$ of $\hkX$.
Applying the same construction as in the proof of Lemma \ref{lm:orien_rev_some_end},  
we obtain an involution $g$ of $\sphere$ that preserves $\hkX$ and $\sgX$ and fixes $\sigma_a,\sigma_f$, and preserves $\sigma_m$ (see Fig.\ \ref{fig:twoone_reflection}). 
 
By the incompressibility of $\partial_f X$, $\inducedsgX$ 
is non-trivial, so $g$ cannot fix $\sigma_m$, and thereby $g\vert_{\sigma_m}$ is orientation-reversing with two fixed points $v_1,v_2$, one of which, say 
$v_1$ is $\sigma_m\cap \sigma_f$.
The fixed point set $\Stwo$, which is a $2$-sphere, of $g$ cuts the knot $(\sphere,\sigma_m)$ into two arc-ball pairs; both are trivial by the atoroidality of $\inducedhkX$.

Lastly, $\sigma_a$ cuts $\Stwo$ into two disks, and $v_1$ and $v_2$ are not in the same disk, for otherwise $\sigma_a$ bounds in $\Stwo$ a disk whose interior disjoint from $\sgX$, contradicting the incompressibility of $\partial_f X$. As a result, $\sigma_a\cup \sigma_m\subset \sphere$ is a Hopf link, and $\sigma_f$ is its tunnel, so $(\sphere,\sgX)$ is the spatial graph $\twoone$.  
\end{proof}
\begin{lemma}\label{lm:g_two_orien_rev_some_end}
Suppose $g=2$, and $\Compl\HK$ is atoroidal.
If $f(X)=X$, for some end $X$ and $f\in \nAut{\sphere,\HK}$, then $\pair\simeq \fourone$.
\end{lemma}
\begin{proof}
By Lemma \ref{lm:orien_rev_some_end}, $\partial_f X_1=\partial_f X_2$, and hence by Lemma \ref{lm:g_two_orien_rev_one_end}, 
$\pair$ is a regular neighborhood of 
the $3$-connected sum of $\twoone$.  
\end{proof}

\begin{theorem}\label{teo:genus_two_symmetry}
Suppose $g=2$, and $\Compl\HK$ is atoroidal.
Then
\[\pSym \HK<\Z_2,\quad \Sym{\HK}<\Z_2\oplus \Z_2,\]
and $\Sym{\HK}\simeq\Z_2\oplus \Z_2$ 
if and only if $\pair$ is $\fourone$.   
\end{theorem}
\begin{proof} 
There are two ends in a maximal $\p 3$-decomposition, and 
every mapping class $[f]\in\pSym\HK$ 
induces a permutation on the set of ends, and hence
there is a homomorphism 
$\pSym\HK\rightarrow \Z_2$, which is injective by Lemma \ref{lm:orien_pre_some_end}, so 
$\pSym\HK<\Z_2$. 
On the other hand, for every orientation-reversing homeomorphism $f$, $f^2$ is orientation-preserving and preserves both ends, so $f^2\simeq \id$. This implies $\Sym\HK<\Z_2\oplus\Z_2$.
For the third assertion, we note first the symmetry group of $\fourone$ is $\Z_2\oplus \Z_2$ by \cite{Wan:24}. On the other hand, if $\Sym\HK\simeq \Z_2\oplus \Z_2$, then there exists an orientation-reversing 
homeomorphism $f$ preserving $X_1,X_2$; therefore $\pair\simeq \fourone$ 
by Lemma \ref{lm:g_two_orien_rev_some_end}.
\end{proof}

\subsubsection*{Examples}
Here we construct examples realizing
the symmetry groups in Theorem \ref{teo:genus_two_symmetry}. 
Consider the family of handlebody-knots $(\sphere,V_{m,n})$ in Fig.\ \ref{fig:Vmn} with $m,n$ indicating the numbers of full twists. By Corollary \ref{cor:genus_two_b_irre}, $\Compl{V_{m,n}}$ is $\partial$-irreducible.

Now, consider the $3$-decomposing surface $P\subset\Compl{V_{m,n}}$ in Fig.\ \ref{fig:Vmn}, and note that components of
$\partial P$ are parallel in $\partial V_{m,n}$ and hence $P$ is a $\p 3$-surface. Observe also that $P$ decomposes
$\Compl{V_{m,n}}$ into two components
$X,Y$, which induce two spatial graphs 
$\inducedsgX,\inducedsgY$ (see Fig.\ \ref{fig:inducedsg_Vmn}) whose constituent links are non-split. 
Thus $P$ is incompressible, and 
$\{P\}$ is a maximal $\p 3$-system. 
%
In addition, $X,Y$ both are handlebodies and hence atoroidal, 
so $\Compl{V_{m,n}}$ 
is atoroidal by Lemma \ref{lm:tori}. This implies $(\sphere,V_{m,n})$ is $\p 2$-indecomposable and hence $2$-indecomposable, given $g=2$.

If $\vert m\vert=\vert n\vert=1$, then $\pairVmn\simeq\fourone$; if $\vert m\vert\neq \vert n\vert$, then $\Sym{V_{m,n}}=\mathbf{1}$. If $m=n\neq \pm 1$ (resp.\ $m=-n\neq \pm 1$), $\pairVmn$ is chiral (resp.\ amphichiral) with $\Sym{V_{m,n}}\simeq \Z_2$. Thus all possibilities in Theorem \ref{teo:genus_two_symmetry} are realized. 

Theorem \ref{teo:genus_two_symmetry} can be used to detect handlebody-knot chirality. For instance, the positive symmetry group of 
$\sixten$ has a non-trivial mapping class given by the rotation that swaps the two $\p 3$-surfaces $P_1,P_2$; see Fig.\ \ref{fig:sixten}. Since $\sixten\not\simeq\fourone$, we have $\sixten$ is chiral.

\begin{figure}[h]
\begin{subfigure}[b]{.34\linewidth}
\centering
\begin{overpic}[scale=.12,percent]{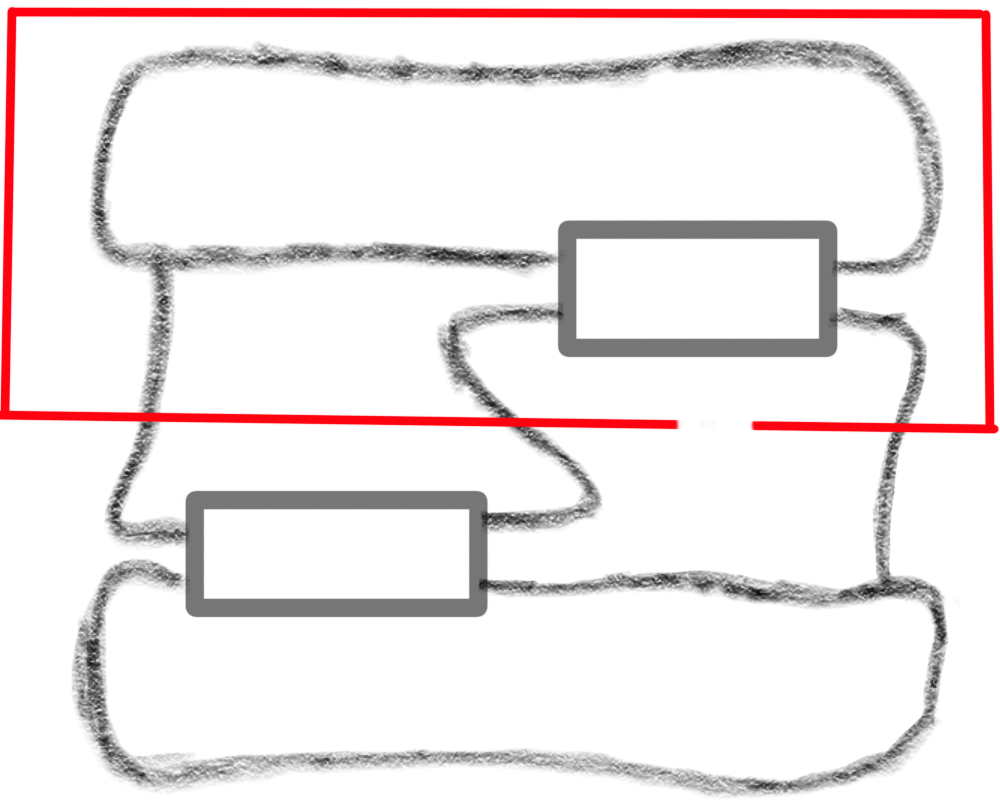}
\put(69,36){$P$}
\put(2.5,71){$X$}´
\put(1,1){$Y$}
\put(67,49){\Large $m$}
\put(31,22.5){\Large $n$}
\end{overpic}
\caption{Family $\pairVmn$.}
\label{fig:Vmn}
\end{subfigure}
\begin{subfigure}[b]{.3\linewidth}
\centering
\begin{overpic}[scale=.08,percent]{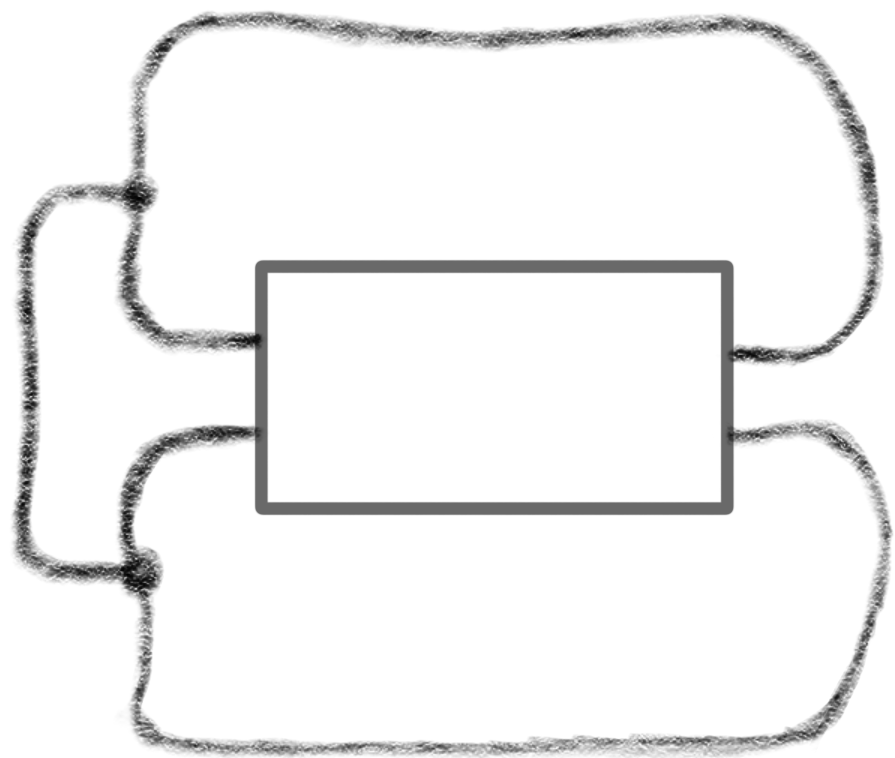}
\put(51,37){\Large $k$}
\end{overpic}
\caption{$k=m$ or $n$.}
\label{fig:inducedsg_Vmn}
\end{subfigure}
\begin{subfigure}[b]{.34\linewidth}
\centering
\begin{overpic}[scale=.12,percent]{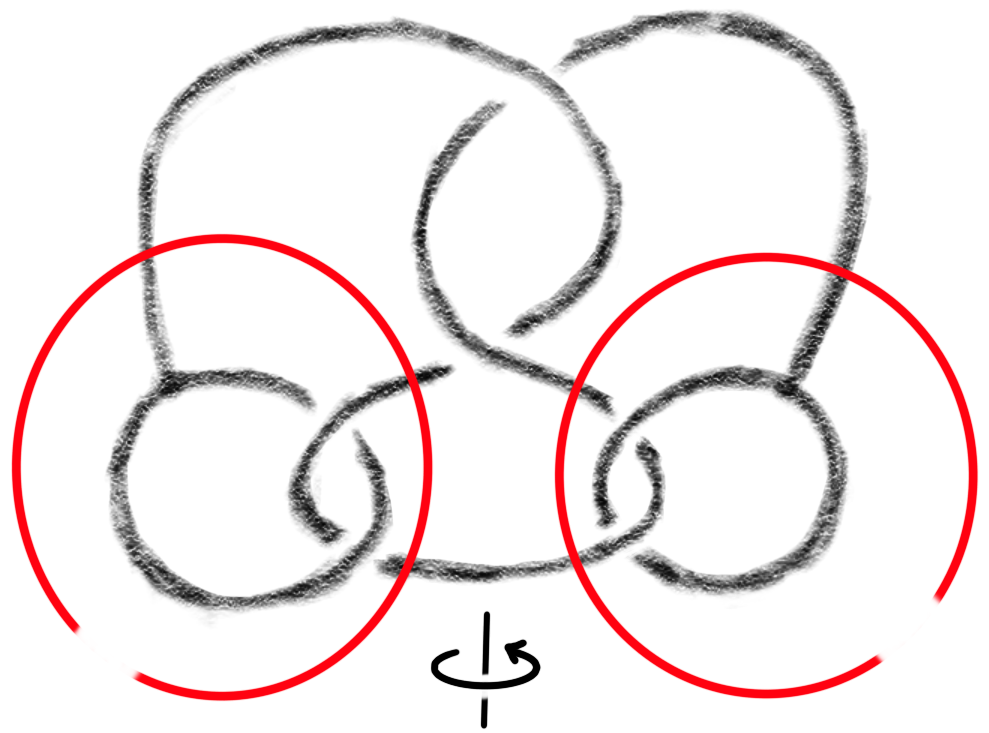}
\put(90,5){$P_2$}
\put(5,3){$P_1$}
\end{overpic}
\caption{Symmetry of $\sixten$.}
\label{fig:sixten}
\end{subfigure}
\caption{}
\end{figure} 

\subsection{Higher genus classification}\label{subsubsec:higher_genus_classification}
\begin{theorem}\label{teo:higher_genus_symmetry}
If $g>2$, then $\vert H\vert$ divides $e$. 
\end{theorem}
\begin{proof}
By Lemma \ref{lm:orien_pre_some_end}, $H$ acts freely on the set of ends. The assertion then follows from Burnside's Lemma.
\end{proof}

\begin{corollary}\label{cor:odd_prime_symmtery}
If $\Compl\HK$ is atoroidal and $e=g>2$ is prime, then $\Sym\HK=\pSym\HK$, and it is either trivial or $\mathbb{Z}_g$. 
\end{corollary}
\begin{proof}
By Lemma \ref{lm:finiteness}, $\Sym\HK$ 
is finite, and 
by Theorem \ref{teo:higher_genus_symmetry}, $\vert\Sym\HK\vert$ divides $e$. 
Since $e$ is an odd prime and the mapping class of an orientation-reversing homeomorphism has an even order, the assertion follows. 
\end{proof}
 
\subsubsection*{Examples}
Recall that the genus $g>2$ handlebody-knot $\pairwheel$ in Fig.\ \ref{fig:pairwheel}, which satisfies the assumption of Section \ref{subsec:symmetry} by Lemma \ref{lm:example_birre_atoro_twoindecom}. 
By Theorem \ref{teo:higher_genus_symmetry}, the rotation $r$ by $\frac{2\pi}{g}$ against the axis perpendicular to the paper generates $\Sym{\wheel}$. 

\end{document}